\documentclass[11pt,a4paper]{amsart}
\usepackage{amsfonts}
\usepackage{}
\usepackage{amsfonts}
\usepackage{amsmath,xypic}
\usepackage{mathrsfs}
\usepackage{amssymb}

\usepackage{CJK,CJKnumb}
\usepackage[CJKbookmarks,colorlinks,
            linkcolor=black,
            anchorcolor=black,
            citecolor=blue]{hyperref}
\usepackage{color}              
\usepackage{indentfirst}        
\usepackage{latexsym,bm}        
\usepackage{amsmath,amssymb}    
\usepackage{pstricks}
\usepackage{pst-node}
\usepackage{pst-tree}
\usepackage{pst-plot}
\usepackage{pst-text}
\usepackage{graphicx}
\usepackage{cases}
\usepackage{pifont}
\usepackage{txfonts}
\usepackage[all,knot,poly]{xy}

\allowdisplaybreaks[4]  

\CJKtilde   

\newcommand{\R}{\mathscr{R}}

\newtheorem{theorem}{Theorem}[section]
\newtheorem{lemma}{Lemma}[section]
\newtheorem{remark}{Remark}[section]
\newtheorem{proposition}{Proposition}[section]
\newtheorem{definition}{Definition}[section]

\allowdisplaybreaks[4]
                               { \end{spacing}
                               } 
\begin{document}
\title[Double-bosonization and Majid's Conjecture, (III)]{Double-bosonization and Majid's Conjecture, (III):
 type-crossing and inductions of $E_6$ and $E_7$, $E_8$}
\author[H. Hu]{Hongmei Hu}
\address{Department of Mathematics,  Shanghai Key Laboratory of Pure Mathematics and Mathematical Practice,
East China Normal University,
Minhang Campus,
Dong Chuan Road 500,
Shanghai 200241,
PR China}
\email{hmhu0124@126.com}

\author[N. Hu]{Naihong Hu$^\ast$}
\address{Department of Mathematics,  Shanghai Key Laboratory of Pure Mathematics and Mathematical Practice,
East China Normal University,
Minhang Campus,
Dong Chuan
Road 500,
Shanghai 200241,
PR China}
\email{nhhu@math.ecnu.edu.cn}

\thanks{$^\ast$N.~H., supported by the NSFC (Grant No.
 11271131).}

\date{}
\maketitle

\newcommand*{\abstractb}[3]{ %
                             \begingroup%
                             \leftskip=8mm \rightskip=8mm
                             \fontsize{11pt}{\baselineskip}\noindent{\textbf{Abstract} ~}{#1}\\ 
                             {\textbf{Keywords} ~}{#2}\\ 
                             {\textbf{MR(2010) Subject Classification} ~}{#3}\\ 
                                                          \endgroup
                           }
\begin{abstract}
Double-bosonization construction in Majid \cite{majid1}
is expectedly allowed to generate a tree of quantum groups. Some main branches of the tree in \cite{HH1, HH2} have been depicted how to grow up.
This paper continues to elucidate the type-crossing and inductive constructions of exceptional quantum groups of types $E_6$ and $E_7$, $E_8$, respectively,
based on the generalized double-bosonization Theorem established in \cite{HH2}.
Thus the Majid's expectation for the inductive constructions of $U_q(\mathfrak g)$'s for all finite-dimensional complex simple Lie algebras is completely achieved.
\end{abstract}

\section{Introduction and our results}
A striking feature of quantum group theory is in close connection with many branches of mathematics and physics,
such as Lie groups, Lie algebras and their representations,
Hecke algebras and their representation theory,
quantum invariants theory of knots or links and $3$-manifolds, as well as the current hot studies on monoidal categories and various categorifications, etc.
So quantum group theory always attracts many mathematicians to find some better ways in a suitable framework to construct their structures defined initially by generators and relations.
In the early nineties,
Ringel \cite{ringel} realized the positive part of quantum groups by quiver representations and Hall algebras,
which inspired Lusztig's canonical base theory \cite{lus1,lus2}.
Bridgeland \cite{Bri} realized the entire quantum groups of type $ADE$ by the Ringel-Hall algebras.
In \cite{rosso},
Rosso also realized the positive part of $U_q(\mathfrak g)$ by the quantum shuffle in a braided category,
and gave a recipe on inductive constructions in the sense of quantum shuffle.
The axiomatic construction for the entire quantum groups for the standard Drinfeld-Jimbo types (cf. Fang-Rosso in \cite{fang})
has been nontrivially generalized to the multi-parameter setting by Hu-Li-Rosso \cite{hlr}.
On the other hand,
Majid \cite{majid1} (also Sommerh\"auser in \cite{somm}) constructed $U_q(\mathfrak g)$ by establishing the double-bosonization theory in a braided category of $H$-modules (resp. in a larger braided category, i.e., the category of Yetter-Drinfeld $H$-modules).
In this paper,
we mainly focus on the double-bosonization method in \cite{majid1}.

Associated with any mutually dual braided groups $B^{\star},B$ covariant under a
background quasitriangular Hopf algebra $H$,
there is a new quantum group structure on the tensor space $B^{\star}\otimes H\otimes B$ by double-bosonization in \cite{majid1},
consisting of $H$ extended by $B$ as additional `positive roots'
and its dual $B^{\star}$ as additional `negative roots'.
The construction is more direct than the quantum double because one can reach to $U_q(\mathfrak g)$
(rather than to take a quotient from the quantum double).
Specially,
Majid regarded $U_{q}(\mathfrak n^{\pm})$ as the mutually dual braided groups in the braided category of left $H$-modules for `Cartan subalgebra' $H$,
then recovered $U_q(\mathfrak g)$ by double-bosonization.
On the other hand, Majid claimed that the double-bosonization construction allowed to be used to generate a tree of quantum groups. That is, many new quantum groups and the inductive constructions of $U_q(\mathfrak g)$'s for all complex simple Lie algebras $\mathfrak g$ were supposed to be obtained by this theory in \cite{majid1}.
He thought that at each node of the tree, there are many choices to adjoin a pair of certain braided groups covariant under the quantum group at that node.
This is a representation-theoretical challenge to elaborate the full tree structure. After a few examples given by Majid himself \cite{majid1, majid6} twenty years ago, recently, 
almost main branches have been depicted in \cite{HH1, HH2, HH3}.
The remaining situation left is the $E$-series. This consists of the aim of this paper.

Nowadays,
Majid's framework on braided groups in certain braided categories has been developed into the framework on Nichols algebras in the Yetter-Drinfeld categories for the purpose to classify finite-dimensional pointed Hopf algebras; see the works \cite{AS1, AS2} of Andruskiewitsch and Schneider,  and the one \cite{HS} of Heckenberger and Schneider.
Hence, we believe in some sense that it is significant to study what kinds of new finite-dimensional Hopf algebras can be found via the
double-bosonization procedure.
An interesting application of our constructions might be connected to a recent work of Cuntz and Lentner on Nichols algebras (also see the concluding remarks in \cite{G} for further interesting considerations).

The paper is organized as follows.
In section 2, we recall some basic facts about FRT-bialgebras associated with $R$-matrices in \cite{FRT1},
Majid's double bosonization theorem in \cite{majid1} and the generalized double-bosonization construction theorem established in \cite{HH2} suitable for those irregular $R$-matrices we encountered when we treat the exceptional cases (beyond the considerations in \cite{majid1}).
In view of some results and methods in \cite{HH1,HH2},
we construct $U_{q}(E_{6})$ by working on $U_{q}(D_{5})$ and its one of the half-spin representations,
$U_{q}(E_{7})$, $U_{q}(E_{8})$ can be constructed from $U_{q}(E_{6})$, $U_{q}(E_{7})$ and their minimal dimensional fundamental representations in section 3.
These representations are given in diagrams (see Figures 1 to 3).
In the last section,
we give a proposition to summarize our inductive constructions of $U_q(\mathfrak g)$'s for all complex finite-dimensional simple Lie algebras,
and also reflect these constructions on the Dynkin diagrams.
With these results, we see that both the double-bosonization theory and the quantum shuffle theory (due to Rosso \cite{rosso}) yield the same tree structure of $U_q(\mathfrak g)$'s for all complex semisimple Lie algebras.

\section{Preliminaries}
In this paper, let $k$ be complex field,
$\mathbb{R}$ real field,
$E$ the Euclidean space $\mathbb{R}^{n}$ or its a suitable subspace.
$\varepsilon_{i}$ denotes the usual orthogonal unit vectors in $\mathbb{R}^{n}$.
$\mathfrak g$ is a finite-dimensional complex semisimple Lie algebra with simple roots $\alpha_{i}$.
$\lambda_{i}$ is the fundamental weight corresponding to simple root $\alpha_{i}$.
Cartan matrix of $\mathfrak g$ is $(a_{ij})$,
where $a_{ij}=\frac{2(\alpha_{i},\alpha_{j})}{(\alpha_{i},\alpha_{i})}$,
and $d_{i}=\frac{(\alpha_{i},\alpha_{i})}{2}$.
Let $(H,\R)$ be a quasitriangular Hopf algebra,
where $\R$ is called the universal $R$-matrix,
$\R=\R^{(1)}\otimes \R^{(2)}$,
$\R_{21}=\R^{(2)}\otimes \R^{(1)}$,
we denote by $\Delta,\eta,\epsilon$ its coproduct,
counit,
unit,
and by $S$ its antipode.

We shall use Sweedler's notation:
for $h \in H$,
$\Delta(h)=h_{1}\otimes h_{2}$.
$H^{\text{op}} \ (H^{\text{cop}})$ denotes the opposite (co)algebra structure of $H$, respectively.
$\mathfrak{M}_{H} \ ({}_{H}\mathfrak{M}$) denotes the braided category consisting of right (left) $H$-modules, respectively.
If there is a coquasitriangular Hopf algebra $A$ such that dual pair $(H,A)$ is a weakly quasitriangular,
then $\mathfrak{M}_{H} \ ({}_{H}\mathfrak{M}$) is equivalent to the braided category ${}^{A}\mathfrak{M} \ (\mathfrak{M}^{A})$ of left (right) $A$-comodules, respectively.
For the detailed description of these theories,
we left to the readers to refer to Drinfeld's and Majid's papers \cite{dri}, \cite{majid2}, \cite{majid3}, and so on.
Braided group is a braided bialgebra or Hopf algebra in some braided category,
in order to distinguish from the ordinary Hopf algebras, let
$\underline{\Delta},\, \underline{S}$ denote its coproduct and antipode, respectively.

\subsection{Majid's double-bosonization theory}
Let $C, B$ be a pair of braided groups in $\mathfrak{M}_{H}$, which
are called dually paired if there is an intertwiner
$\text{ev}: C\otimes B \longrightarrow k$ such that
$$
\text{ev}(cd,b)=\text{ev}(d,b_{\underline{(1)}})\text{ev}(c,b_{\underline{(2)}}), \quad
\text{ev}(c,ab)=\text{ev}(c_{\underline{(2)}},a)\text{ev}(c_{\underline{(1)}},b),\quad\forall a,b\in B,c,d\in C.
$$
Then $C^{\text{op}/\text{cop}}$(with opposite product and coproduct) is a Hopf algebra in $_{H}\mathfrak{M}$,
which is dual to $B$ in the sense of an ordinary duality pairing $\langle~,~\rangle$,
which is $H$-bicovariant: $\langle h\rhd c,b\rangle=\langle c, b\lhd h\rangle,$ for all $h\in H$.
Let $\overline{C}=(C^{\text{op}/\text{cop}})^{\underline{\text{cop}}}$,
then $\overline{C}$ is a braided group in $_{\overline{H}}\mathfrak{M}$,
where $\overline{H}$ is $(H,\R_{21}^{-1})$.
With these,
Majid gave the following double bosonization theorem and some results in \cite{majid1}:
\begin{theorem}\label{ml1}$(${\rm\bf Majid}$)$
On the tensor space $\bar{C}\otimes H \otimes B$,
there is a unique Hopf algebra structure $U=U(\bar{C},H,B)$ such that
$H\ltimes B$ and $\bar{C}\rtimes H$ are sub-Hopf algebras by the canonical inclusions and
$$
\left.
\begin{array}{rl}
bc=&(\R_{1}^{(2)}\rhd c_{\overline{(2)}})
\R_{2}^{(2)}\R_{1}^{-(1)}(b_{\underline{(2)}}\lhd\R_{2}^{-(1)})\,
\langle\R_{1}^{(1)}\rhd c_{\overline{(1)}},b_{\underline{(1)}}\lhd
\R_{2}^{(1)}\rangle\,\cdot\\
&\qquad\qquad\quad\cdot\,\langle\R_{1}^{-(2)}\rhd \overline{S}c_{\overline{(3)}},b_{\underline{(3)}}\lhd \R_{2}^{-(2)}\rangle,
\end{array}
\right.
$$
for all $ b\in B, \,c\in\overline{C}$ viewed in $U$.
Here $\R_{1},\R_{2}$
are distinct copies of the quasitriangular structure $\R$ of $H$.
The product, coproduct of $U$ are given by
$$
\left.
\begin{array}{rl}
(c\otimes h\otimes b)\cdot(d\otimes g\otimes a)=&c(h_{(1)}\R_{1}^{(2)}\rhd d_{\overline{(2)}})
\otimes h_{(2)}\R_{2}^{(2)}\R_{1}^{-(1)}g_{(1)}\otimes (b_{\underline{(2)}}\lhd\R_{2}^{-(1)}g_{(2)})a\\
&\langle\R_{1}^{(1)}\rhd d_{\overline{(1)}},b_{\underline{(1)}}\lhd\R_{2}^{(1)}\rangle
\langle\R_{1}^{-(2)}\rhd\overline{S}d_{\overline{(3)}},b_{\underline{(3)}}\lhd\R_{2}^{-(2)}\rangle;
\end{array}
\right.
$$
$$
\Delta(c\otimes h\otimes b)=c_{\overline{(1)}}\otimes \R^{-(1)}h_{(1)}\otimes b_{\underline{(1)}}\lhd\R^{(1)}
\otimes \R^{-(2)}\rhd c_{\overline{(2)}}\otimes h_{(2)}\R^{(2)}\otimes b_{\underline{(2)}}.
$$
\end{theorem}
\begin{remark}\label{Cross}
If there exists a coquasitriangular Hopf algebra $A$ such that $(H,A)$ is a weakly quasitriangular dual pair,
and
$b,
c$
are primitive elements,
then some relations simplify to
$$
[b,c]=\R(b^{\overline{(1)}})\langle c,b^{\overline{(2)}}\rangle-
\langle c^{\overline{(1)}},b\rangle\bar{\R}(c^{\overline{(2)}});
$$
$$
\Delta b=b^{\overline{(2)}}\otimes \R(b^{\overline{(1)}})+1\otimes b,\quad
\Delta c=c\otimes 1+\bar{\R}(c^{\overline{(2)}})
\otimes c^{\overline{(1)}}.
$$
\end{remark}

\subsection{Generalized double-bosonization construction theorem}
Let $R$ be an invertible matrix obeying the quantum Yang-Baxter equation.
There is a bialgebra $A(R)$ \cite{FRT1} corresponding to the $R$-matrix,
called an $FRT$-bialgebra.
\begin{definition}
$A(R)$ is generated by $1$ and $t=\{t^{i}_{j}\}$,
having the following structure:
$$R^{i}_{a}{}^{j}_{b}t^{b}_{k}t^{a}_{l}=t^{i}_{a}t^{j}_{b~}R^{a}_{l}{}^{b}_{k},\quad
~~\Delta(t^{i}_{j})=t^{i}_{a}\otimes t^{a}_{j},\quad
\epsilon(t^{i}_{j})=\delta^{i}_{j}.$$
$A(R)$ is a coquasitriangular bialgebra with
$\R:\, A(R)\otimes A(R) \longrightarrow k$ such that $\R(t^{i}_{j}\otimes t^{k}_{l})=R^{ik}_{jl}$.
Here
$T_{1}=T\otimes \text{I}$,
$T_{2}=\text{I}\otimes T$, $R^{ik}_{jl}$ denotes the entry at row $(ik)$ and column $(jl)$ in $R$.
\end{definition}
By double cross product of bialgebra  $A(R)$,
Majid obtained bialgebra $\widetilde{U(R)}$ in \cite{majid5},
which is generated by $m^{\pm}$ and
satisfies
\begin{gather*}
Rm^{\pm}_{1}m^{\pm}_{2}=m^{\pm}_{2}m^{\pm}_{1}R,
\quad
Rm_{1}^{+}m_{2}^{-}=m_{2}^{-}m_{1}^{+}R,
\\
\Delta((m^{\pm})^{i}_{j})=(m^{\pm})_{j}^{a}\otimes (m^{\pm})_{a}^{i},
\quad
\epsilon((m^{\pm})^{i}_{j})=\delta_{ij},
\end{gather*}
where
$(m^{\pm})^{i}_{j}$ denotes the entry at row $i$ and column $j$ in $m^{\pm}$.
Moreover,
$(\widetilde{U(R)},A(R))$ is a weakly quasitriangular dual pair of bialgebras.
In braided category ${}^{A(R)}\mathfrak{M} \ (\mathfrak{M}^{A{(R)}})$,
there are two classical braided groups $V(R',R)$, $V^{\vee}(R',R_{21}^{-1})$ in \cite{majid4},
called braided (co)vector algebras, respectively.
\begin{proposition}\label{prop1}
Suppose that $R^{\prime}$ is another matrix such that
$
(i)~R_{12}R_{13}R^{\prime}_{23}=R^{\prime}_{23}R_{13}R_{12},
$
$
R_{23}R_{13}R^{\prime}_{12}=R^{\prime}_{12}R_{13}R_{23},
$
$(ii)~(PR+1)(PR^{\prime}-1)=0$,
$(iii)~R_{21}R^{\prime}_{12}=R^{\prime}_{21}R_{12}$,
where $P$ is a permutation matrix with the entry $P^{ij}_{kl}=\delta_{il}\delta_{jk}$.
Then the braided-vectors algebra $V(R',R)$ defined by generators $1$,
$\{e^{i}~|~i=1,\cdots,n\}$,
 and relations
$e^{i}e^{j}=\sum\limits_{a,b} R'{}^{ji}_{ab}e^{a}e^{b}$
forms a braided group with
$\underline{\Delta}(e^{i})=e^{i}\otimes 1+1\otimes e^{i},
\underline{\epsilon}(e^{i})=0,
\underline{S}(e^{i})=-e^{i},
\Psi(e^{i}\otimes e^{j})=\sum\limits_{a,b}R^{ji}_{ab}e^{a}\otimes e^{b}$
in braided category ${}^{A(R)}\mathfrak{M}$.
Under the duality $\langle f_{j},e^{i}\rangle=\delta_{ij}$,
braided-covectors algebra $V^{\vee}(R',R_{21}^{-1})$ defined by $1$ and $\{f_{j}~|~j=1,\cdots,n\}$,
and relations
$f_{i}f_{j}=\sum\limits_{a,b}f_{b}f_{a}R'{}^{ab}_{ij}$
forms another braided group with
$\underline{\Delta}(f_{i})=f_{i}\otimes 1+1\otimes f_{i},$
$\underline{\epsilon}(f_{i})=0,$
$\underline{S}(f_{i})=-f_{i},$
$\Psi(f_{i}\otimes f_{j})=\sum\limits_{a,b}f_{b}\otimes f_{a}R^{ab}_{ij}$
in braided category $\mathfrak{M}^{A(R)}$.
\end{proposition}
\begin{remark}\label{rem1}
There is a way to find suitable matrix $R'$ for a given $R$ in \cite{majid4}.
We know that $PR$ obeys some minimal polynomial
$
\prod_{i}(PR-x_{i})=0.
$
For each nonzero eigenvalue $x_{i}$,
we can normalise $R$ so that $x_{i}=-1$.
then set
$
R'=P+P\prod\limits_{j\neq i}(PR-x_{j}),
$
which satisfies the conditions $(i)$---$(iii)$.
It gives us at least one braided (co)vector algebra for each nonzero eigenvalue of $PR$.
\end{remark}

With these,
starting from some irreducible representation $T_{V}$ of $U_{q}(\mathfrak g)$ and its corresponding $R$-matrix $R_{VV}$.
Then there exists a unique dual pairing $\langle \cdot,\cdot\rangle$ between $U_{q}^{\text{ext}}(\mathfrak g)$ and $H_{R_{VV}}$ in \cite{HH2}
such that
$$
\langle (m^{+})^{i}_{j},t^{k}_{l}\rangle=R_{VV}{}^{ik}_{jl}=\lambda R^{ik}_{jl},\quad
\langle (m^{-})^{i}_{j},t^{k}_{l}\rangle=R_{VV}^{-1}{}^{ki}_{lj}=\lambda^{-1}(R^{-1}){}^{ki}_{lj},
$$
$$
\langle (m^{+})^{i}_{j},\tilde{t}^{k}_{l}\rangle
=(R_{VV}^{t_{2}})^{-1}{}^{ik}_{jl}=((\lambda R)^{t_{2}})^{-1}{}^{ik}_{jl},\quad
\langle (m^{-})^{i}_{j},\tilde{t}^{k}_{l}\rangle
=[(R^{-1}_{VV})^{t_{1}}]^{-1}{}^{lj}_{ki}=[(\lambda^{-1}R^{-1})^{t_{1}}]^{-1}{}^{lj}_{ki}.
$$
Such $\lambda$ is called a {\it normalization constant of quantum groups}.
The convolution-invertible algebra\,/\,anti-coalgebra maps $\R,\bar{\R}$ in
$\textrm{hom}(H_{\lambda R},U_{q}^{\textrm{ext}}(\mathfrak{g}))$ are
$$
\R(t^{i}_{j})=(m^{+})^{i}_{j},\quad
\R(\tilde{t}^{i}_{j})=(m^{+})^{-1}{}^{j}_{i};\quad
\bar{\R}(t^{i}_{j})=(m^{-})^{i}_{j},\quad
\bar{\R}(\tilde{t}^{i}_{j})=(m^{-})^{-1}{}^{j}_{i},
$$
where $U_{q}^{\text{ext}}(\mathfrak g)$ is the FRT-form of $U_{q}(\mathfrak{g})$ or extended quantized enveloping algebra
adjoined by the elements $K_{i}^{\frac{1}{m}}$,
$H_{\lambda R}$ is the Hopf algebra associated with $A(R_{VV})$,
generated by $t^{i}_{j}, \tilde{t}^{i}_{j}$ and subject to
certain relations.
With the weakly quasitriangular dual pair of Hopf algebras in the above Theorem,
we obtain that the braided (co)vector algebras
$V(R^{\prime},R)\in{}^{H_{\lambda R}}\mathfrak{M}$
and
$V^{\vee}(R^{\prime},R_{21}^{-1})\in \mathfrak{M}^{H_{\lambda R}}$.
However, in order to yield the required quantum group of higher rank $1$, we need to centrally extend the pair $(U_{q}^{\text{ext}}(\mathfrak g),H_{\lambda R})$
to the pair
$$
\Bigl(\,\widetilde{U_q^{\text{ext}}(\mathfrak
g)}=U_{q}^{\text{ext}}(\mathfrak g)\otimes k[c,c^{-1}],
\,\widetilde{H_{\lambda R}}=H_{\lambda R}\otimes k[g,g^{-1}]\,\Bigr)
$$
with the action $e^{i}\lhd c=\lambda e^{i}$,
$f_{i}\lhd c=\lambda f_{i}$ and the extended pair
$\langle c,g\rangle=\lambda$. In this way, we easily see that
the braided algebras $V(R^{\prime},R)\in{}^{\widetilde{H_{\lambda R}}}\mathfrak{M}$
and
$V^{\vee}(R^{\prime},R_{21}^{-1})\in \mathfrak{M}^{\widetilde{H_{\lambda R}}}$.

Having these,
we have the following {\it generalized double-bosonization construction Theorem} in \cite{HH2}.
\begin{theorem}\label{cor1}
Let $R_{VV}$ be the $R$-matrix corresponding to the irreducible representation $T_{V}$ of $U_{q}(\mathfrak{g})$.
There exists a normalization constant $\lambda$ such that $\lambda R=R_{VV}$.
Then the new quantum group
$U=U(V^{\vee}(R^{\prime},R_{21}^{-1}),\widetilde{U_{q}^{ext}(\mathfrak g)},V(R^{\prime},R))$ has the following the cross relations:
\begin{gather*}
cf_{i}=\lambda f_{i}c,\quad
e^{i}c=\lambda ce^{i},\quad
[c,m^{\pm}]=0,\quad
[e^{i},f_{j}]=\delta_{ij}\frac{(m^{+})^{i}_{j}c^{-1}-c(m^{-})^{i}_{j}}{q_{\ast}-q_{\ast}^{-1}};\\
e^{i}(m^{+})^{j}_{k}=R_{VV}{}^{ji}_{ab}(m^{+})^{a}_{k}e^{b},\quad
(m^{-})^{i}_{j}e^{k}=R_{VV}{}^{ki}_{ab}e^{a}(m^{-})^{b}_{j},\\
(m^{+})^{i}_{j}f_{k}=f_{b}(m^{+})^{i}_{a}R_{VV}{}^{ab}_{jk},\quad
f_{i}(m^{-})^{j}_{k}=(m^{-})^{j}_{b}f_{a}R_{VV}{}^{ab}_{ik},
\end{gather*}
and the coproduct: $$\Delta c=c\otimes c, \quad \Delta
e^{i}=e^{a}\otimes (m^{+})^{i}_{a}c^{-1}+1\otimes e^{i}, \quad
\Delta f_{i}=f_{i}\otimes 1+c(m^{-})^{a}_{i}\otimes f_{a},$$ and the
counit $\epsilon e^{i}=\epsilon f_{i}=0$.
We can normalize $e^{i}$ such that the factor
$q_{\ast}-q_{\ast}^{-1}$ meets the situation
we need.
\end{theorem}
\begin{remark}
The $R$-matrix in \cite{majid1} is the $P\circ\cdot\circ P$ of the ordinary $R$-matrix,
so the quantized enveloping algebras constructed by double-bosonization satisfy
$
E_{i}K_{j}=q_{i}^{a_{ij}}K_{j}E_{i},
$
$
F_{i}K_{j}=q_{i}^{-a_{ij}}K_{j}F_{i},
$
$
[E_{i},F_{j}]=\delta_{ij}\frac{K_{i}-K_{i}^{-1}}{q_i-q_i^{-1}},
$
where $q_{i}=q^{d_{i}}$, in addition to the same $q$-Serre relations.
So it is necessary to notice that we use this definition of $U_{q}(\mathfrak g)$ in our proof.
\end{remark}
\begin{remark}\label{R}
The $R$-matrix datum associated with a given module $V$ comes from the $R_{VV}$-matrix,
defined by
$R_{VV}=B_{VV}\circ(T_{V}\otimes T_{V})(\mathfrak{R})$.
Here
$
\mathfrak{R}=\sum\limits_{r_{1},\cdots,r_{n}=0}^{\infty}\prod\limits_{j=1}^{n}
\frac{(1-q_{\beta_{j}}^{-2})^{r_{j}}}{[r_{j}]_{q_{\beta_{j}}}}
q_{\beta_{j}}^{\frac{r_{j}(r_{j}+1)}{2}}E_{\beta_{j}}^{r_{j}}\otimes F_{\beta_{j}}^{r_{j}}
$
is main part of the universal $R$-matrix of $U_{q}(\mathfrak g)$.
$B_{VV}$ denotes the linear operator on $V\otimes V$ given by
$B_{VV}(v\otimes w):=q^{(\mu,\mu^{\prime})}v\otimes w$
for $v\in V_{\mu}$, $w\in V_{\mu^{\prime}}$.
We can endow every basis with weights indices
such that the operators $T_{V}(E_{i})$ and $T_{V}(F_{i})$ are raising and lowering the indices of the basis,
then the required $R_{VV}$-matrix is upper triangular (see \cite{HH2}).
\end{remark}
\begin{remark}\label{remm}
The matrix $m^{\pm}$ can be obtained by $S(l^{\pm}_{ij})=(m^{\pm})^{i}_{j}$,
where $l^{\pm}_{ij}$ is the $L^{\pm}$-functionals of $U_{q}(\mathfrak g)$.
The $L^{\pm}$-functionals of $U_{q}^{ext}(\mathfrak g)$ can obtained from $U_{h}(\mathfrak g)$'s $L^{\pm}$-functionals in Chapter 8 in \cite{klim}, defined by
\begin{gather*}
(id\otimes T_{V})(\mathcal{R})(1\otimes v_{j})=\sum\limits_{i}l_{ij}^{+}\otimes v_{i},\\
(T_{V}\otimes id)(\mathcal{R}^{-1})(v_{j}\otimes 1)=(T_{V}\otimes id)(S\otimes id)(\mathcal{R})(v_{j}\otimes 1)=\sum\limits_{i}v_{i}\otimes l_{ij}^{-},
\end{gather*}
where
$$
\mathcal{R}=exp(h\sum\limits_{i,j}B_{ij}(H_{i}\otimes H_{j}))\prod\limits_{\beta\in \Delta_{+}}exp_{q_{\beta}}((1-q_{\beta}^{-2})(E_{\beta}\otimes F_{\beta}))
$$
is the universal $R$-matrix of $U_{h}(\mathfrak g)$.
Here matrix $(B_{ij})$ is the inverse of matrix $(C_{ij})=(d_{j}^{-1}a_{ij})$,
the $q$-exponential function $\text{exp}_{q}x$ is defined by $\text{exp}_{q}x=\sum\limits_{r=0}^{\infty}\frac{q^{\frac{r(r+1)}{2}}}{[r]_{q}!}x^{r}.$
\end{remark}

\section{Exceptional quantized enveloping algebras $U_q(\mathfrak g)$ of type $E$ series}
From the type-crossing constructions of $U_{q}(G_{2})$ and $U_{q}(F_{4})$ in \cite{HH2},
we know that it is most crucial to find suitable representations,
and get the $R_{VV}$-matrices we need,
then to achieve the required data $R,R^{\prime}$ via determining the minimal polynomial of the $PR_{VV}$-matrix.

\subsection{Type-crossing construction of $U_{q}(E_{6})$}
In order to construct $U_{q}(E_{6})$,
it is natural to start from $U_{q}(D_{5})$ and a certain representation of it.
We choose the $16$-dimensional half-spin representation $T_{V}$ of $U_{q}(D_{5})$,
which is given by the Figure $1$,
where $(s_{1},s_{2},s_{3},s_{4},s_{5})$ denotes the basis of weight $\frac{1}{2}(s_{1}\varepsilon_{1}+s_{2}\varepsilon_{2}+s_{3}\varepsilon_{3}+s_{4}\varepsilon_{4}+s_{5}\varepsilon_{5})$,
$s_{i}=\pm$.
\begin{center}
\setlength{\unitlength}{1mm}
\begin{picture}(70,115)
\put(0,0){Figure $1$.~~The $16$-dimensional half-spin representation of $U_{q}(D_{5})$}
\put(50,4){(--~--~--~--~+)}
\put(68,7){$f_{1}$}
\put(57.5,5.5){\line(0,1){10}}
\put(58,10){$\alpha_{4}$}
\put(50,16){(--~--~--~+~--)}
\put(68,19){$f_{2}$}
\put(57.5,17.5){\line(0,1){10}}
\put(58,21){$\alpha_{3}$}
\put(50,27){(--~--~+~--~--)}
\put(68,29){$f_{3}$}
\put(58.5,28.5){\line(2,1){16}}
\put(69,32){$\alpha_{2}$}
\put(65,37){(--~+~--~--~--)}
\put(83,39){$f_{4}$}
\put(45,36){\line(5,-3){12}}
\put(44,32){$\alpha_{5}$}
\put(35,37){(--~--~+++)}
\put(53,39){$f_{6}$}
\put(45.5,38){\line(0,1){10}}
\put(40,43){$\alpha_{2}$}
\put(75,38){\line(0,1){10}}
\put(75,43){$\alpha_{1}$}
\put(65,49){(+~--~--~--~--)}
\put(83,51){$f_{5}$}
\put(62,59.5){\line(5,-3){14}}
\put(73,54){$\alpha_{5}$}
\put(35,49){(--~+~--~++)}
\put(53,51){$f_{7}$}
\put(52,48){\line(5,-3){13}}
\put(57,45){$\alpha_{5}$}
\put(45.5,50.5){\line(5,3){15}}
\put(30,59){\line(5,-3){14}}
\put(30,55){$\alpha_{3}$}
\put(48.5,59){$\alpha_{1}$}
\put(30,61){\line(5,3){14}}
\put(30,64){$\alpha_{1}$}
\put(45,69){\line(5,-3){14}}
\put(55,64){$\alpha_{3}$}
\put(44.5,71){\line(0,1){9}}
\put(45,74){$\alpha_{2}$}
\put(15,69){\line(5,-3){14}}
\put(16,64){$\alpha_{4}$}
\put(15.5,70){\line(0,1){9}}
\put(11,75){$\alpha_{1}$}
\put(53,59){(+~--~--~++)}
\put(71,61){$f_{8}$}
\put(20,59){(--~++--~+)}
\put(38,61){$f_{9}$}
\put(35,69){(+~--~+~--~+)}
\put(54,71){$f_{10}$}
\put(5,69){(--~+++~--)}
\put(23,71){$f_{12}$}
\put(35,79){(++~--~--~+)}
\put(54,81){$f_{11}$}
\put(5,79){(+~--~++~--)}
\put(23,81){$f_{13}$}
\put(22,78){\line(5,-3){13}}
\put(28,75){$\alpha_{4}$}
\put(14,81){\line(5,3){14}}
\put(14,86){$\alpha_{2}$}
\put(30,89){\line(5,-3){14}}
\put(36,86){$\alpha_{4}$}
\put(20,89){(++~--~+~--)}
\put(38,91){$f_{14}$}
\put(28,90){\line(0,1){9}}
\put(28.5,94){$\alpha_{3}$}
\put(28,101){\line(0,1){9}}
\put(28.5,106){$\alpha_{5}$}
\put(20,100){(+++~--~--)}
\put(38,103){$f_{15}$}
\put(20,110){(+++++)}
\put(38,113){$f_{16}$}
\end{picture}
\end{center}

The actions of the $E_{i}$, $F_{i}$'s can be read directly from the Figure $1$,
$K_{j}$'s actions can be obtained by the defining relations of $K_{i}$ and $E_{i},F_{i}$ in $U_{q}(D_{5})$.
Since $E_{\beta}(f_{m})=f_{n} \Longleftrightarrow F_{\beta}(f_{n})=f_{m}$
and
$E_{i}^{2}$ is zero action,
we can prove that the corresponding matrix $PR_{VV}$ is symmetrizable,
which is similar to situations in \cite{HH2}.
Moreover,
for the corresponding $R$-matrix $R_{VV}$,
we have
\begin{lemma}\label{leme6}
\begin{enumerate}
\item
The minimal polynomial of $PR_{VV}$ is
$$(P{R}_{VV}+q^{-\frac{3}{4}}I)(P{R}_{VV}-q^{\frac{5}{4}}I)(P{R}_{VV}-q^{-\frac{3}{4}}I)=0.$$
\item Those entries we need in matrices $m^{\pm}$ are as follows:

\begin{itemize}
\item In $m^{+}$:

$\left.
\begin{array}{ll}
(m^{+})^{1}_{2}=-(q-q^{-1})E_{4}K_{1}^{\frac{1}{2}}K_{2}K_{3}^{\frac{3}{2}}K_{4}^{\frac{1}{4}}K_{5}^{\frac{3}{4}},&
(m^{+})^{2}_{2}=K_{1}^{\frac{1}{2}}K_{2}K_{3}^{\frac{3}{2}}K_{4}^{\frac{1}{4}}K_{5}^{\frac{3}{4}},
\\
(m^{+})^{2}_{3}=-(q-q^{-1})E_{3}K_{1}^{\frac{1}{2}}K_{2}K_{3}^{\frac{1}{2}}K_{4}^{\frac{1}{4}}K_{5}^{\frac{3}{4}},&
(m^{+})^{3}_{3}=K_{1}^{\frac{1}{2}}K_{2}K_{3}^{\frac{1}{2}}K_{4}^{\frac{1}{4}}K_{5}^{\frac{3}{4}},
\\
(m^{+})^{3}_{4}=-(q-q^{-1})E_{2}K_{1}^{\frac{1}{2}}K_{3}^{\frac{1}{2}}K_{4}^{\frac{1}{4}}K_{5}^{\frac{3}{4}},&
(m^{+})^{4}_{4}=K_{1}^{\frac{1}{2}}K_{3}^{\frac{1}{2}}K_{4}^{\frac{1}{4}}K_{5}^{\frac{3}{4}},
\\
(m^{+})^{4}_{5}=-(q-q^{-1})E_{1}K_{1}^{-\frac{1}{2}}K_{3}^{\frac{1}{2}}K_{4}^{\frac{1}{4}}K_{5}^{\frac{3}{4}},&
(m^{+})^{5}_{5}=K_{1}^{-\frac{1}{2}}K_{3}^{\frac{1}{2}}K_{4}^{\frac{1}{4}}K_{5}^{\frac{3}{4}},
\\
(m^{+})^{15}_{16}=-(q-q^{-1})E_{5}K_{1}^{-\frac{1}{2}}K_{2}^{-1}K_{3}^{-\frac{3}{2}}K_{4}^{-\frac{3}{4}}K_{5}^{-\frac{5}{4}},&
(m^{+})^{16}_{16}=K_{1}^{-\frac{1}{2}}K_{2}^{-1}K_{3}^{-\frac{3}{2}}K_{4}^{-\frac{3}{4}}K_{5}^{-\frac{5}{4}}.
\end{array}
\right.
$
\item In $m^{-}$:

$
\left.
\begin{array}{ll}
(m^{-})^{2}_{1}=q(q-q^{-1})K_{1}^{-\frac{1}{2}}K_{2}^{-1}K_{3}^{-\frac{3}{2}}K_{4}^{-\frac{1}{4}}K_{5}^{-\frac{3}{4}}F_{4},&
(m^{-})^{2}_{2}=K_{1}^{-\frac{1}{2}}K_{2}^{-1}K_{3}^{-\frac{3}{2}}K_{4}^{-\frac{1}{4}}K_{5}^{-\frac{3}{4}},
\\
(m^{-})^{13}_{12}=q(q-q^{-1})K_{1}^{\frac{1}{2}}K_{3}^{\frac{1}{2}}K_{4}^{\frac{3}{4}}K_{5}^{\frac{1}{4}}F_{1},&
(m^{-})^{13}_{13}=K_{1}^{\frac{1}{2}}K_{3}^{\frac{1}{2}}K_{4}^{\frac{3}{4}}K_{5}^{\frac{1}{4}},
\\
(m^{-})^{14}_{13}=q(q-q^{-1})K_{1}^{\frac{1}{2}}K_{2}K_{3}^{\frac{1}{2}}K_{4}^{\frac{3}{4}}K_{5}^{\frac{1}{4}}F_{2},&
(m^{-})^{14}_{14}=K_{1}^{\frac{1}{2}}K_{2}K_{3}^{\frac{1}{2}}K_{4}^{\frac{3}{4}}K_{5}^{\frac{1}{4}},
\\
(m^{-})^{15}_{14}=q(q-q^{-1})K_{1}^{\frac{1}{2}}K_{2}K_{3}^{\frac{3}{2}}K_{4}^{\frac{3}{4}}K_{5}^{\frac{1}{4}}F_{3},&
(m^{-})^{15}_{15}=K_{1}^{\frac{1}{2}}K_{2}K_{3}^{\frac{3}{2}}K_{4}^{\frac{3}{4}}K_{5}^{\frac{1}{4}},
\\
(m^{-})^{16}_{15}=q(q-q^{-1})K_{1}^{\frac{1}{2}}K_{2}K_{3}^{\frac{3}{2}}K_{4}^{\frac{3}{4}}K_{5}^{\frac{5}{4}}F_{5},&
(m^{-})^{16}_{16}=K_{1}^{\frac{1}{2}}K_{2}K_{3}^{\frac{3}{2}}K_{4}^{\frac{3}{4}}K_{5}^{\frac{5}{4}}.
\end{array}
\right.
$
\end{itemize}
\end{enumerate}
\end{lemma}
\begin{proof}
Those entries in $(2)$ of Lemma \ref{leme6}
can be obtained by Remark \ref{remm}.
Then we will focus on the minimal polynomial in $(1)$.
The decomposition of tensor product for this spin representation is
$T_{V}\otimes T_{V}=V_{1}\oplus V_{2}\oplus V_{3}$,
where
$V_{i}, i=1,2,3$ is the irreducible representation with highest weight
$\varepsilon_{1}+\varepsilon_{2}+\varepsilon_{3}+\varepsilon_{4}+\varepsilon_{5},
\varepsilon_{1}+\varepsilon_{2}+\varepsilon_{3},
\varepsilon_{1}$, respectively.
This means that matrix $PR_{VV}$ has only $3$ distinct eigenvalues,
denoted by $x_{1},x_{2},x_{3}$.
Then $P{R}_{VV}$ obeys the minimal polynomial
$(P{R}_{VV}-x_{1}I)(P{R}_{VV}-x_{2}I)(P{R}_{VV}-x_{3}I)=0$.
We use an ingenious method to get these eigenvalues $x_{i}$.

According to Remark \ref{R},
we obtain the nonzero entries of $PR_{VV}-x_{i}I$ in rows of $(12),(21)$ are
\begin{equation*}
\begin{split}
(P{R}_{VV}-x_{i}I)^{1}_{1}{}^{2}_{2}&=-x_{i},\\
(P{R}_{VV}-x_{i}I)^{1}_{2}{}^{2}_{1}&=q^{\frac{1}{4}},\\
(P{R}_{VV}-x_{i}I)^{2}_{1}{}^{1}_{2}&=q^{\frac{1}{4}},\\
(P{R}_{VV}-x_{i}I)^{2}_{2}{}^{1}_{1}&=q^{\frac{1}{4}}(q-q^{-1})-x_{i}.
\end{split}
\end{equation*}
Then we obtain that the 
nonzero entries in row $(12)$ of $(P{R}_{VV}-x_{1}I)(P{R}_{VV}-x_{2}I)$  still lie in columns $(12)$ and $(21)$,
given by
$$
[(P{R}_{VV}-x_{1}I)(P{R}_{VV}-x_{2}I)]{}^{1}_{1}{}^{2}_{2}
=(P{R}_{VV}-x_{1}I)^{1}_{1}{}^{2}_{2}(P{R}_{VV}-x_{2}I)^{1}_{1}{}^{2}_{2}
+(P{R}_{VV}-x_{1}I)^{1}_{2}{}^{2}_{1}(P{R}_{VV}-x_{2}I)^{2}_{1}{}^{1}_{2},
$$
$$
[(P{R}_{VV}-x_{1}I)(P{R}_{VV}-x_{2}I)]{}^{1}_{2}{}^{2}_{1}
=(P{R}_{VV}-x_{1}I)^{1}_{1}{}^{2}_{2}(P{R}_{VV}-x_{2}I)^{1}_{2}{}^{2}_{1}
+(P{R}_{VV}-x_{1}I)^{1}_{2}{}^{2}_{1}(P{R}_{VV}-x_{2}I)^{2}_{2}{}^{1}_{1}.
$$
With these, some entries in row $(12)$ of $(P{R}_{VV}-x_{1}I)(P{R}_{VV}-x_{2}I)(P{R}_{VV}-x_{3}I)$ can be obtained by
\begin{equation*}
\begin{split}
&[(P{R}_{VV}-x_{1}I)(P{R}_{VV}-x_{2}I)(P{R}_{VV}-x_{3}I)]^{1}_{1}{}^{2}_{2}\\
&\quad =[(P{R}_{VV}-x_{1}I)(P{R}_{VV}-x_{2}I)]^{1}_{1}{}^{2}_{2}(P{R}_{VV}-x_{3}I)^{1}_{1}{}^{2}_{2}\\
&\qquad+[(P{R}_{VV}-x_{1}I)(P{R}_{VV}-x_{2}I)]^{1}_{2}{}^{2}_{1}(P{R}_{VV}-x_{3}I)^{2}_{1}{}^{1}_{2}\\
&\quad=-x_{1}x_{2}x_{3}-q^{\frac{1}{2}}(x_{1}+x_{2}+x_{3})+q^{\frac{3}{4}}(q-q^{-1}),
\\
&[(P{R}_{VV}-x_{1}I)(P{R}_{VV}-x_{2}I)(P{R}_{VV}-x_{3}I)]^{1}_{2}{}^{2}_{1}\\
&\quad=[(P{R}_{VV}-x_{1}I)(P{R}_{VV}-x_{2}I)]^{1}_{1}{}^{2}_{2}(P{R}_{VV}-x_{3}I)^{1}_{2}{}^{2}_{1}\\
&\qquad+[(P{R}_{VV}-x_{1}I)(P{R}_{VV}-x_{2}I)]^{1}_{2}{}^{2}_{1}(P{R}_{VV}-x_{3}I)^{2}_{2}{}^{1}_{1}\\
&\quad=q^{\frac{1}{4}}(x_{1}x_{2}+x_{2}x_{3}+x_{1}x_{3})-q^{\frac{1}{2}}(q-q^{-1})(x_{1}+x_{2}+x_{3})
+q^{\frac{3}{4}}(q-q^{-1})^{2}+q^{\frac{3}{4}}.
\end{split}
\end{equation*}

We consider the nonzero entries in rows $(1,16),(16,1)$ as follows
\begin{equation*}
\begin{split}
(P{R}_{VV}-x_{i}I)^{1,}_{1,}{}^{16}_{16}&=-x_{i},\\
(P{R}_{VV}-x_{i}I)^{1,}_{16,}{}^{16}_{1}&=q^{-\frac{3}{4}},\\
(P{R}_{VV}-x_{i}I)^{16,}_{16,}{}^{1}_{1}&=-x_{i},\\
(P{R}_{VV}-x_{i}I)^{16,}_{1,}{}^{1}_{16}&=q^{-\frac{3}{4}},\\
(P{R}_{VV}-x_{i}I)^{16,}_{a_{0},}{}^{1}_{b_{0}}&=Y_{a_{0}}
\end{split}
\end{equation*}
for some columns $(a_{0},b_{0})$.
If there exists $Y_{a_{0}}\neq 0$,
then
$
1<a_{0}, \, b_{0}<16,
$
so
$$(P{R}_{VV}-x_{i}I)^{a_{0},}_{1,}{}^{b_{0}}_{16}=0.$$
Nonzero entries in row $(1,16)$ of $(P{R}_{VV}-x_{1}I)(P{R}_{VV}-x_{2}I)$ are 
given by
\begin{equation*}
\begin{split}
&[(P{R}_{VV}-x_{1}I)(P{R}_{VV}-x_{2}I)]^{1,}_{1,}{}^{16}_{16}\\
&\quad=(P{R}_{VV}-x_{1}I)^{1,}_{1,}{}^{16}_{16}(P{R}_{VV}-x_{2}I)^{1,}_{1,}{}^{16}_{16}
+(P{R}_{VV}-x_{1}I)^{1,}_{16,}{}^{16}_{1}(P{R}_{VV}-x_{2}I)^{16,}_{1,}{}^{1}_{16},
\\
&[(P{R}_{VV}-x_{1}I)(P{R}_{VV}-x_{2}I)]^{1,}_{16,}{}^{16}_{1}\\
&\quad=(P{R}_{VV}-x_{1}I)^{1,}_{1,}{}^{16}_{16}(P{R}_{VV}-x_{2}I)^{1,}_{16,}{}^{16}_{1}
+(P{R}_{VV}-x_{1}I)^{1,}_{16,}{}^{16}_{1}(P{R}_{VV}-x_{2}I)^{16,}_{16,}{}^{1}_{1},
\\
&[(P{R}_{VV}-x_{1}I)(P{R}_{VV}-x_{2}I)]^{1,}_{a_{0},}{}^{16}_{b_{0}}\\
&\quad=(P{R}_{VV}-x_{1}I)^{1,}_{1,}{}^{16}_{16}(P{R}_{VV}-x_{2}I)^{1,}_{a_{0},}{}^{16}_{b_{0}}
+(P{R}_{VV}-x_{1}I)^{1,}_{16,}{}^{16}_{1}(P{R}_{VV}-x_{2}I)^{16,}_{a_{0},}{}^{1}_{b_{0}}\\
&\quad=(P{R}_{VV}-x_{1}I)^{1,}_{16,}{}^{16}_{1}(P{R}_{VV}-x_{2}I)^{16,}_{a_{0},}{}^{1}_{b_{0}}.
\end{split}
\end{equation*}

With these,
the entry $[(P{R}_{VV}-x_{1}I)(P{R}_{VV}-x_{2}I)(P{R}_{VV}-x_{3}I)]^{1,}_{1,}{}^{16}_{16}$ is
\begin{equation*}
\begin{split}
&[(P{R}_{VV}-x_{1}I)(P{R}_{VV}-x_{2}I)(P{R}_{VV}-x_{3}I)]^{1,}_{1,}{}^{16}_{16}\\
&\quad=[(P{R}_{VV}-x_{1}I)(P{R}_{VV}-x_{2}I)]^{1,}_{1,}{}^{16}_{16}(P{R}_{VV}-x_{3}I)^{1,}_{1,}{}^{16}_{16}\\
&\qquad+[(P{R}_{VV}-x_{1}I)(P{R}_{VV}-x_{2}I)]^{1,}_{16,}{}^{16}_{1}(P{R}_{VV}-x_{3}I)^{16,}_{1,}{}^{1}_{16}\\
&\qquad+[(P{R}_{VV}-x_{1}I)(P{R}_{VV}-x_{2}I)]^{1,}_{a_{0},}{}^{16}_{b_{0}}(P{R}_{VV}-x_{3}I)^{a_{0},}_{1,}{}^{b_{0}}_{16}\\
&\quad=[(P{R}_{VV}-x_{1}I)(P{R}_{VV}-x_{2}I)]^{1,}_{1,}{}^{16}_{16}(P{R}_{VV}-x_{3}I)^{1,}_{1,}{}^{16}_{16}\\
&\qquad+[(P{R}_{VV}-x_{1}I)(P{R}_{VV}-x_{2}I)]^{1,}_{16,}{}^{16}_{1}(P{R}_{VV}-x_{3}I)^{16,}_{1,}{}^{1}_{16}\\
&\quad=-x_{1}x_{2}x_{3}-q^{-\frac{3}{2}}(x_{1}+x_{2}+x_{3}).
\end{split}
\end{equation*}

Then we have the following equations

$
\left\{
\begin{array}{l}
-x_{1}x_{2}x_{3}-q^{\frac{1}{2}}(x_{1}+x_{2}+x_{3})+q^{\frac{3}{4}}(q-q^{-1})=0,\\
q^{\frac{1}{4}}(x_{1}x_{2}+x_{2}x_{3}+x_{1}x_{3})-q^{\frac{1}{2}}(q-q^{-1})(x_{1}+x_{2}+x_{3})
+q^{\frac{3}{4}}(q-q^{-1})^{2}+q^{\frac{3}{4}}=0,\\
-x_{1}x_{2}x_{3}-q^{-\frac{3}{2}}(x_{1}+x_{2}+x_{3})=0.
\end{array}
\right.
$

Solving these equations,
we obtain that these $3$ eigenvalues are
$q^{\frac{5}{4}},q^{-\frac{3}{4}},-q^{-\frac{3}{4}}$.
So the minimal polynomial of $PR_{VV}$ is
$(P{R}_{VV}+q^{-\frac{3}{4}}I)(P{R}_{VV}-q^{\frac{5}{4}}I)(P{R}_{VV}-q^{-\frac{3}{4}}I)=0.$
\end{proof}
From the above Lemma \ref{leme6},
setting
$$R=q^{\frac{3}{4}}R_{VV},\quad
R^{\prime}=RPR-(q^{2}+1)R+(q^{2}+1)P,$$
then we obtain $(PR+I)(PR^{\prime}-I)=0$.
With these, we have the following
\begin{theorem}
With the quantum normalization constant $\lambda=q^{-\frac{3}{4}}$,
identify $e^{16}, f_{16}$, and $(m^{+})^{16}_{16}c^{-1}$ with the additional simple root vectors $E_{6},F_{6}$, and the group-like element $K_{6}$, respectively,
then the new quantum group $U(V^{\vee}(R^{\prime},R_{21}^{-1}),\widetilde{U_{q}^{ext}(D_{5})},V(R^{\prime},R))$
is the quantized enveloping algebra $U_{q}(E_{6})$ with $K_{i}^{\frac{1}{4}}$ adjoined.
\end{theorem}
\begin{proof}
Under the above identification,
$[E_{6},F_{6}]=\frac{K_{6}-K_{6}^{-1}}{q-q_{-1}}$,
$\Delta(E_{6})=E_{6}\otimes K_{6}+1\otimes E_{6}$,
and
$\Delta(F_{6})=F_{6}\otimes 1+K_{6}^{-1}\otimes F_{6}$
can be obtained directly by Theorem \ref{cor1}.
We find that the minor diagonal entries $(m^{-})^{i+1}_{i}$ has the form
$(m^{-})^{i+1}_{i}=q(q-q^{-1})(m^{-})^{i+1}_{i+1}F$.
Combing with $(m^{-})^{i+1}_{i}e^{16}=\lambda R^{16}_{16}{}^{i+1}_{i+1}e^{16}(m^{-})^{i+1}_{i}$
and $(m^{-})^{i+1}_{i+1}e^{16}=\lambda R^{16}_{16}{}^{i+1}_{i+1}e^{16}(m^{-})^{i+1}_{i+1}$ given by the equalities in Theorem \ref{cor1},
we obtain $[E_{6},F_{j}]=0,$
$1\leq j\leq 5$,
namely,
$[E_{6},F_{i}]=\delta_{ij}\frac{K_{6}-K_{6}^{-1}}{q-q_{-1}}$.
According to the cross relations in Theorem \ref{cor1} again,
we have

$
\left.
\begin{array}{l}
e^{16}(m^{+})^{j}_{j}=\lambda R^{j}_{16}{}^{j}_{16}(m^{+})^{j}_{j}e^{16},\\
(m^{+})^{i+1}_{i+1}K_{5-i}=(m^{+})^{i}_{i}, \quad
1\leq i\leq 4,\\
R^{i}_{16}{}^{i}_{16}=1,\quad
1\leq i\leq 5,\quad
R^{16}_{16}{}^{16}_{16}=q^{2}
\end{array}
\right\}
\Longrightarrow
\left\{
\begin{array}{l}
E_{6}K_{i}=K_{i}E_{6}, \quad 1\leq i\leq 4,\\
E_{6}K_{5}=q^{-1}K_{5}E_{6},\\
E_{6}K_{6}=q^{2}K_{6}E_{6}.
\end{array}
\right.
$

The relations between $K_{6}$ and others $E_{i}$'s are given by

$
E_{1}K_{6}
=E_{1}K_{1}^{-\frac{1}{2}}K_{2}^{-1}K_{3}^{-\frac{3}{2}}K_{4}^{-\frac{3}{4}}K_{5}^{-\frac{5}{4}}c^{-1}
=q^{-1}qK_{1}^{-\frac{1}{2}}K_{2}^{-1}K_{3}^{-\frac{3}{2}}K_{4}^{-\frac{3}{4}}K_{5}^{-\frac{5}{4}}c^{-1}E_{1}
=K_{6}E_{1}.
$

$
E_{2}K_{6}=E_{2}K_{1}^{-\frac{1}{2}}K_{2}^{-1}K_{3}^{-\frac{3}{2}}K_{4}^{-\frac{3}{4}}K_{5}^{-\frac{5}{4}}c^{-1}
=q^{\frac{1}{2}}q^{-2}q^{\frac{3}{2}}K_{1}^{-\frac{1}{2}}K_{2}^{-1}K_{3}^{-\frac{3}{2}}K_{4}^{-\frac{3}{4}}K_{5}^{-\frac{5}{4}}c^{-1}E_{2}
=K_{6}E_{2},
$

$
E_{3}K_{6}=E_{3}K_{1}^{-\frac{1}{2}}K_{2}^{-1}K_{3}^{-\frac{3}{2}}K_{4}^{-\frac{3}{4}}K_{5}^{-\frac{5}{4}}c^{-1}
=qq^{-3}q^{\frac{3}{4}}q^{\frac{5}{4}}K_{1}^{-\frac{1}{2}}K_{2}^{-1}K_{3}^{-\frac{3}{2}}K_{4}^{-\frac{3}{4}}K_{5}^{-\frac{5}{4}}c^{-1}E_{3}
=K_{6}E_{3},
$

$
E_{4}K_{6}=E_{4}K_{1}^{-\frac{1}{2}}K_{2}^{-1}K_{3}^{-\frac{3}{2}}K_{4}^{-\frac{3}{4}}K_{5}^{-\frac{5}{4}}c^{-1}
=q^{\frac{1}{2}}q^{-2}q^{\frac{3}{2}}K_{1}^{-\frac{1}{2}}K_{2}^{-1}K_{3}^{-\frac{3}{2}}K_{4}^{-\frac{3}{4}}K_{5}^{-\frac{5}{4}}c^{-1}E_{4}
=K_{6}E_{4},
$

$
E_{5}K_{6}=E_{5}K_{1}^{-\frac{1}{2}}K_{2}^{-1}K_{3}^{-\frac{3}{2}}K_{4}^{-\frac{3}{4}}K_{5}^{-\frac{5}{4}}c^{-1}
=q^{\frac{3}{2}}q^{-\frac{5}{2}}K_{1}^{-\frac{1}{2}}K_{2}^{-1}K_{3}^{-\frac{3}{2}}K_{4}^{-\frac{3}{4}}K_{5}^{-\frac{5}{4}}c^{-1}E_{5}
=q^{-1}K_{6}E_{5}.
$

We will explore the $q$-Serre relations between $E_{6}$ and others $E_{i}$'s.
$E_{i}$'s appear in the minor entries of $m^{+}$,
so we consider the cross relations related to $(m^{+})^{i}_{i+1}$ in Theorem \ref{cor1}.
$
\left.
\begin{array}{l}
(m^{+})^{i}_{i+1}=-(q-q^{-1})E_{5-i}(m^{+})^{i+1}_{i+1},\\
e^{16}(m^{+})^{i}_{i+1}=\lambda (m^{+})^{i}_{i+1}e^{16},\\
e^{16}(m^{+})^{i+1}_{i+1}=\lambda (m^{+})^{i+1}_{i+1}e^{16},
1\leq i\leq 4,\\
e^{16}(m^{+})^{15}_{16}=\lambda (q(m^{+})^{15}_{16}e^{16}+q(q-q^{-1})(m^{+})^{16}_{16}e^{15},
\end{array}
\right\}
\Longrightarrow
\left\{
\begin{array}{l}
E_{6}E_{i}-E_{i}E_{6}=0,\\
e^{15}=qe^{16}E_{5}-E_{5}e^{16}.
\end{array}
\right.
$

So we need to know the relations between $e^{15}$ and $e^{16},E_{5}$. Indeed,
$e^{15}e^{16}=R^{\prime}{}^{16,}_{a,}{}^{15}_{b}e^{a}e^{b}$ $=(2q^{2}+1)e^{15}e^{16}-2qe^{16}e^{15}
\Longrightarrow~e^{16}e^{15}=qe^{15}e^{16}$,
then combing with the above equality $e^{15}=qe^{16}E_{5}-E_{5}e^{16}$,
we obtain
$$
(E_{6})^{2}E_{5}-[2]_{q}E_{6}E_{5}E_{6}+E_{5}(E_{6})^{2}=0.
$$
On the other hand,
the relation of $e^{15}$ and $E_{5}$ is given by $e^{15}(m^{+})^{15}_{16}=\lambda R^{15}_{a}{}^{15}_{b}(m^{+})^{a}_{16}e^{b}=q^{\frac{5}{4}}(m^{+})^{15}_{16}e^{15}.$
Putting the expression of $(m^{+})^{15}_{16}$ into the equality,
we get
$e^{15}E_{5}=qE_{5}e^{15}.$
Combing with the above equality $e^{15}=qe^{16}E_{5}-E_{5}e^{16}$ again,
we obtain
$$
(E_{5})^{2}E_{6}-[2]_{q}E_{5}E_{6}E_{5}+E_{6}(E_{5})^{2}=0.
$$

The relations of $f_{16}$ and $F_{i},K_{i}$ can be obtained in a similar way.
\end{proof}

\subsection{Inductive construction of $U_{q}(E_{7})$}
For the construction of $U_{q}(E_{7})$,
we choose the $27$-dimensional minimal fundamental representation of $U_{q}(E_{6})$ given by the following Figure $2$,
which also appeared in \cite{KM, KKM}.
$(s_{1},\cdots s_{i})$ denotes the basis of weight $s_{1}\lambda_{1}+\cdots+s_{i}\lambda_{i}$,
$s_{j}=0,1,\overline{1}$,
where $\overline{1}=-1$.
Corresponding to this representation,
we have the following Lemma \ref{leme7} by the similar analysis.
\begin{lemma}\label{leme7}
$(1)$
$PR_{VV}$ obeys the minimal polynomial
$$(P{R}_{VV}-q^{\frac{4}{3}}I)(P{R}_{VV}-q^{-\frac{2}{3}}I)(P{R}_{VV}+q^{-\frac{2}{3}}I)=0.$$
$(2)$
The entries we need in $m^{\pm}$ are obtained by Remark \ref{remm},
\begin{itemize}
\item In $m^{+}:$

$
\left.
\begin{array}{ll}
(m^{+})^{1}_{2}=(q^{-1}-q)E_{1}K_{1}^{\frac{1}{3}}K_{2}K_{3}^{\frac{5}{3}}K_{4}^{2}K_{5}^{\frac{4}{3}}K_{6}^{\frac{2}{3}},&
(m^{+})^{2}_{2}=K_{1}^{\frac{1}{3}}K_{2}K_{3}^{\frac{5}{3}}K_{4}^{2}K_{5}^{\frac{4}{3}}K_{6}^{\frac{2}{3}},
\\
(m^{+})^{2}_{3}=(q^{-1}-q)E_{3}K_{1}^{\frac{1}{3}}K_{2}K_{3}^{\frac{2}{3}}K_{4}^{2}K_{5}^{\frac{4}{3}}K_{6}^{\frac{2}{3}},&
(m^{+})^{3}_{3}=K_{1}^{\frac{1}{3}}K_{2}K_{3}^{\frac{2}{3}}K_{4}^{2}K_{5}^{\frac{4}{3}}K_{6}^{\frac{2}{3}},
\\
(m^{+})^{3}_{4}=(q^{-1}-q)E_{4}K_{1}^{\frac{1}{3}}K_{2}K_{3}^{\frac{2}{3}}K_{4}K_{5}^{\frac{4}{3}}K_{6}^{\frac{2}{3}},&
(m^{+})^{4}_{4}=K_{1}^{\frac{1}{3}}K_{2}K_{3}^{\frac{2}{3}}K_{4}K_{5}^{\frac{4}{3}}K_{6}^{\frac{2}{3}},
\\
(m^{+})^{4}_{5}{=}(q^{-1}-q)E_{2}K_{1}^{\frac{1}{3}}K_{3}^{\frac{2}{3}}K_{4}K_{5}^{\frac{4}{3}}K_{6}^{\frac{2}{3}},&
(m^{+})^{5}_{5}{=}K_{1}^{\frac{1}{3}}K_{3}^{\frac{2}{3}}K_{4}K_{5}^{\frac{4}{3}}K_{6}^{\frac{2}{3}},~~
(m^{+})^{5}_{6}{=}0,
\\
(m^{+})^{4}_{6}=(q^{-1}-q)E_{5}K_{1}^{\frac{1}{3}}K_{2}K_{3}^{\frac{2}{3}}K_{4}K_{5}^{\frac{1}{3}}K_{6}^{\frac{2}{3}},&
(m^{+})^{6}_{6}=K_{1}^{\frac{1}{3}}K_{2}K_{3}^{\frac{2}{3}}K_{4}K_{5}^{\frac{1}{3}}K_{6}^{\frac{2}{3}},
\\
(m^{+})^{6}_{7}=(q^{-1}-q)E_{6}K_{1}^{\frac{1}{3}}K_{2}K_{3}^{\frac{2}{3}}K_{4}K_{5}^{\frac{1}{3}}K_{6}^{-\frac{1}{3}},&
(m^{+})^{7}_{7}=K_{1}^{\frac{1}{3}}K_{2}K_{3}^{\frac{2}{3}}K_{4}K_{5}^{\frac{1}{3}}K_{6}^{-\frac{1}{3}},
\\
(m^{+})^{1}_{1}=K_{1}^{\frac{4}{3}}K_{2}K_{3}^{\frac{5}{3}}K_{4}^{2}K_{5}^{\frac{4}{3}}K_{6}^{\frac{2}{3}},&
(m^{+})^{27}_{27}=K_{1}^{-\frac{2}{3}}K_{2}^{-1}K_{3}^{-\frac{4}{3}}K_{4}^{-2}K_{5}^{-\frac{5}{3}}K_{6}^{-\frac{4}{3}}.
\end{array}
\right.
$

\item In $m^{-}:$

$
\left.
\begin{array}{ll}
(m^{-})^{2}_{1}=\frac{K_{1}^{-\frac{1}{3}}K_{2}^{-1}K_{3}^{-\frac{5}{3}}K_{4}^{-2}K_{5}^{-\frac{4}{3}}K_{6}^{-\frac{2}{3}}F_{1}}{(q^{2}-1)},&
(m^{-})^{2}_{2}=K_{1}^{-\frac{1}{3}}K_{2}^{-1}K_{3}^{-\frac{5}{3}}K_{4}^{-2}K_{5}^{-\frac{4}{3}}K_{6}^{-\frac{2}{3}},
\\
(m^{-})^{3}_{2}=\frac{K_{1}^{-\frac{1}{3}}K_{2}^{-1}K_{3}^{-\frac{2}{3}}K_{4}^{-2}K_{5}^{-\frac{4}{3}}K_{6}^{-\frac{2}{3}}F_{3}}{(q^{2}-1)},&
(m^{-})^{3}_{3}=K_{1}^{-\frac{1}{3}}K_{2}^{-1}K_{3}^{-\frac{2}{3}}K_{4}^{-2}K_{5}^{-\frac{4}{3}}K_{6}^{-\frac{2}{3}},
\\
(m^{-})^{4}_{3}=\frac{K_{1}^{-\frac{1}{3}}K_{2}^{-1}K_{3}^{-\frac{2}{3}}K_{4}^{-1}K_{5}^{-\frac{4}{3}}K_{6}^{-\frac{2}{3}}F_{4}}{(q^{2}-1)},&
(m^{-})^{4}_{4}=K_{1}^{-\frac{1}{3}}K_{2}^{-1}K_{3}^{-\frac{2}{3}}K_{4}^{-1}K_{5}^{-\frac{4}{3}}K_{6}^{-\frac{2}{3}},
\\
(m^{-})^{5}_{4}=\frac{K_{1}^{-\frac{1}{3}}K_{3}^{-\frac{2}{3}}K_{4}^{-1}K_{5}^{-\frac{4}{3}}K_{6}^{-\frac{2}{3}}F_{2}}{(q^{2}-1)},&
(m^{-})^{5}_{5}=K_{1}^{-\frac{1}{3}}K_{3}^{-\frac{2}{3}}K_{4}^{-1}K_{5}^{-\frac{4}{3}}K_{6}^{-\frac{2}{3}},
\\
(m^{-})^{8}_{5}=\frac{K_{1}^{-\frac{1}{3}}K_{3}^{-\frac{2}{3}}K_{4}^{-1}K_{5}^{-\frac{1}{3}}K_{6}^{-\frac{2}{3}}F_{5}}{(q^{2}-1)},&
(m^{-})^{8}_{8}=K_{1}^{-\frac{1}{3}}K_{3}^{-\frac{2}{3}}K_{4}^{-1}K_{5}^{-\frac{1}{3}}K_{6}^{-\frac{2}{3}},
\\
(m^{-})^{7}_{6}=\frac{K_{1}^{-\frac{1}{3}}K_{2}^{-1}K_{3}^{-\frac{2}{3}}K_{4}^{-1}K_{5}^{-\frac{1}{3}}K_{6}^{\frac{1}{3}}F_{6}}{(q^{2}-1)},&
(m^{-})^{7}_{7}=K_{1}^{-\frac{1}{3}}K_{2}^{-1}K_{3}^{-\frac{2}{3}}K_{4}^{-1}K_{5}^{-\frac{1}{3}}K_{6}^{\frac{1}{3}}.
\end{array}
\right.
$
\end{itemize}

\end{lemma}
\begin{proof}
The method of proof is similar to the above Lemma.
For the proof of $(1)$,
the decomposition of tensor product for this representation is
$$(000001)\otimes(000001)=(000002)\oplus(000010)\oplus(100000).$$
Namely,
$V\otimes V=V_{1}\oplus V_{2}\oplus V_{3}$,
where
$V_{i},i=1,2,3$ denotes the irreducible representation with highest weight
$2\lambda_{6},\lambda_{5},\lambda_{1}$, respectively.
This means that
matrix $PR_{VV}$ has $3$ distinct eigenvalues,
denoted by $x_{1},x_{2},x_{3}$,
and set
$\mathcal{N}=(P{R}_{VV}-x_{1}I)(P{R}_{VV}-x_{2}I)(P{R}_{VV}-x_{3}I)$.

We also consider some entries at some special rows in the matrix $\mathcal{N}$.
For example, the nonzero entries at rows $(12)$ and $(21)$ are
$
(P{R}_{VV}-x_{i}I)^{1}_{1}{}^{2}_{2}=-x_{i},
(P{R}_{VV}-x_{i}I)^{1}_{2}{}^{2}_{1}=q^{\frac{1}{3}},
(P{R}_{VV}-x_{i}I)^{2}_{1}{}^{1}_{2}=q^{\frac{1}{3}},
(P{R}_{VV}-x_{i}I)^{2}_{2}{}^{1}_{1}=q^{\frac{1}{3}}(q-q^{-1})-x_{i}.
$
Then we obtain

$
\mathcal{N}^{1}_{1}{}^{2}_{2}
=-x_{1}x_{2}x_{3}-q^{\frac{2}{3}}(x_{1}+x_{2}+x_{3})+q^{2}-1,
$

$
\mathcal{N}^{1}_{2}{}^{2}_{1}
=q^{\frac{1}{3}}(x_{1}x_{2}+x_{1}x_{3}+x_{2}x_{3})
-q^{\frac{2}{3}}(q-q^{-1})(x_{1}+x_{2}+x_{3})+q^{3}+q^{-1}-q.
$

Other special rows we consider are $(1,27)$ and $(27,1)$,
nonzero entries at these rows are
$(P{R}_{VV}-x_{i}I)^{1,}_{1,}{}^{27}_{27}=-x_{i},
(P{R}_{VV}-x_{i}I)^{1,}_{27,}{}^{27}_{1}=q^{-\frac{3}{4}},
(P{R}_{VV}-x_{i}I)^{27,}_{27,}{}^{1}_{1}=-x_{i},
(P{R}_{VV}-x_{i}I)^{27,}_{1,}{}^{1}_{27}=q^{-\frac{3}{4}},
(P{R}_{VV}-x_{i}I)^{27,}_{c_{0},}{}^{1}_{d_{0}}=Y_{c_{0}}
$
for some columns $(a_{0},b_{0})$,
where
$
1<a_{0},b_{0}<27,
$
and
$(P{R}_{VV}-x_{i}I)^{a_{0},}_{1,}{}^{b_{0}}_{27}=0.$
With these entries,
we obtain
$$
\mathcal{N}^{1,}_{1,}{}^{27}_{27}=-x_{1}x_{2}x_{3}-q^{-\frac{4}{3}}(x_{1}+x_{2}+x_{3}).
$$

In view of the value of
$
\mathcal{N}^{1}_{1}{}^{2}_{2},
\mathcal{N}^{1}_{2}{}^{2}_{1},
\mathcal{N}^{1,}_{1,}{}^{27}_{27}
$,
we obtain the following equations

$
\left\{
\begin{array}{l}
-x_{1}x_{2}x_{3}-q^{\frac{2}{3}}(x_{1}+x_{2}+x_{3})+q^{2}-1=0,\\
q^{\frac{1}{3}}(x_{1}x_{2}+x_{1}x_{3}+x_{2}x_{3})
-q^{\frac{2}{3}}(q-q^{-1})(x_{1}+x_{2}+x_{3})+q^{3}+q^{-1}-q=0,\\
-x_{1}x_{2}x_{3}-q^{-\frac{4}{3}}(x_{1}+x_{2}+x_{3})=0.
\end{array}
\right.
$

Solving it, we get that these eigenvalues are
$
q^{\frac{4}{3}},
\pm q^{-\frac{2}{3}}.
$
Then we obtain the minimal polynomial of $PR_{VV}$ is
$(P{R}_{VV}-q^{\frac{4}{3}}I)(P{R}_{VV}-q^{-\frac{2}{3}}I)(P{R}_{VV}+q^{-\frac{2}{3}}I)=0.$
\end{proof}
Setting $R=q^{\frac{2}{3}}R_{VV},R^{\prime}=RPR-(q^{2}+1)R+(q^{2}+1)P$,
we have $(PR+I)(PR^{\prime}-I)=0.$
In view of these results,
we obtain the following
\begin{theorem}
With the quantum normalization constant $\lambda=q^{-\frac{2}{3}}$,
identify $e^{27}, f_{27}$ and $(m^{+})^{27}_{27}c^{-1}$ with the additional simple root vectors $E_{7}, F_{7}$ and the group-element $K_{7}$, respectively,
then the new quantum group
$U(V^{\vee}(R^{\prime},R_{21}^{-1}),\widetilde{U_{q}^{ext}(E_{6})},V(R^{\prime},R))$
is the quantized enveloping algebra $U_{q}(E_{7})$ with $K_{i}^{\frac{1}{3}}$ adjoined.
\end{theorem}
\begin{proof}
$[E_{7},F_{7}]=\frac{K_{7}-K_{7}^{-1}}{q-q_{-1}}$,
$\Delta(E_{7})=E_{7}\otimes K_{7}+1\otimes E_{7}$,
and
$\Delta(F_{7})=F_{7}\otimes 1+K_{7}^{-1}\otimes F_{7}$
can by obtained directly by Theorem \ref{cor1}.
We will explore the relations between $E_{7}$ and $F_{i}$.
$F_{i}$'s appear in the minor entries $(m^{-})^{i+1}_{i}$ of matrix $m^{-}$,
then according to the cross relations in Theorem \ref{cor1},
we obtain

$
\left.
\begin{array}{l}
(m^{-})^{i+1}_{i}=(q^{2}-1)(m^{-})^{i+1}_{i+1}F_{j},\\
(m^{-})^{8}_{5}=(q^{2}-1)(m^{-})^{8}_{8}F_{5},\\
(m^{-})^{i+1}_{i}e^{27}=\lambda R^{27}_{27}{}^{i+1}_{i+1}e^{27}(m^{-})^{i+1}_{i},\\
(m^{-})^{8}_{5}e^{27}=\lambda R^{27}_{27}{}^{8}_{8}e^{27}(m^{-})^{8}_{5},\\
(m^{-})^{i+1}_{i+1}e^{27}=\lambda R^{27}_{27}{}^{i+1}_{i+1}e^{27}(m^{-})^{i+1}_{i+1},
\end{array}
\right\}
\Longrightarrow
E_{7}F_{j}=F_{j}E_{7}
\Longrightarrow
[E_{7},F_{j}]=0,\quad 1\leq j\leq 6.
$

Namely,
$[E_{7},F_{i}]=\delta_{7j}\frac{K_{7}-K_{7}^{-1}}{q-q_{-1}}$.
Under the above identification,
$
E_{7}K_{7}=e^{27}(m^{+})^{27}_{27}c^{-1}=\lambda R^{27}_{27}{}^{27}_{27}(m^{+})^{27}_{27}e^{27}c^{-1}
=R^{27}_{27}{}^{27}_{27}(m^{+})^{27}_{27}c^{-1}e^{27}
=q^{\frac{2}{3}}R_{VV}{}^{27}_{27}{}^{27}_{27}(m^{+})^{27}_{27}c^{-1}e^{27}
=q^{2}(m^{+})^{27}_{27}c^{-1}e^{27}=q^{2}K_{7}E_{7}.
$
According to the cross relation $e^{27}(m^{+})^{i}_{i}=\lambda R^{i}_{i}{}^{27}_{27}(m^{+})^{i}_{i}e^{27}$ in Theorem \ref{cor1},
combining with
$
(m^{+})^{2}_{2}K_{1}=(m^{+})^{1}_{1},
$
$
(m^{+})^{3}_{3}K_{3}=(m^{+})^{2}_{2},
$
$
(m^{+})^{4}_{4}K_{4}=(m^{+})^{3}_{3},
$
$
(m^{+})^{5}_{5}K_{2}=(m^{+})^{4}_{4},
$
$
(m^{+})^{7}_{7}K_{6}=(m^{+})^{5}_{5},
$
and
$R^{i}_{i}{}^{27}_{27}=1,1\leq i\leq 6$,
$
R^{7}_{7}{}^{27}_{27}=q,
$
we obtain the relations between $E_{7}$ and $K_{i}$: 
$
\left\{
\begin{array}{l}
E_{7}K_{i}=K_{i}E_{7}, \ 1\leq i\leq 5,\\
E_{7}K_{6}=q^{-1}K_{6}E_{7}.
\end{array}
\right.
$

The relations between $K_{7}$ and others $E_{i}$'s can be obtained by the following equalities.
$$
E_{1}K_{7}=E_{1}K_{1}^{-\frac{2}{3}}K_{2}^{-1}K_{3}^{-\frac{4}{3}}K_{4}^{-2}K_{5}^{-\frac{5}{3}}K_{6}^{-\frac{4}{3}}c^{-1}
=q^{-\frac{4}{3}}q^{\frac{4}{3}}K_{1}^{-\frac{2}{3}}K_{2}^{-1}K_{3}^{-\frac{4}{3}}K_{4}^{-2}K_{5}^{-\frac{5}{3}}K_{6}^{-\frac{4}{3}}c^{-1}E_{1}
=K_{7}E_{1},
$$
$$
E_{2}K_{7}=E_{2}K_{1}^{-\frac{2}{3}}K_{2}^{-1}K_{3}^{-\frac{4}{3}}K_{4}^{-2}K_{5}^{-\frac{5}{3}}K_{6}^{-\frac{4}{3}}c^{-1}
=q^{-2}q^{2}K_{1}^{-\frac{2}{3}}K_{2}^{-1}K_{3}^{-\frac{4}{3}}K_{4}^{-2}K_{5}^{-\frac{5}{3}}K_{6}^{-\frac{4}{3}}c^{-1}E_{2}
=K_{7}E_{2},
$$
$$
E_{3}K_{7}=E_{3}K_{1}^{-\frac{2}{3}}K_{2}^{-1}K_{3}^{-\frac{4}{3}}K_{4}^{-2}K_{5}^{-\frac{5}{3}}K_{6}^{-\frac{4}{3}}c^{-1}
=q^{\frac{2}{3}}q^{-\frac{8}{3}}q^{2}K_{1}^{-\frac{2}{3}}K_{2}^{-1}K_{3}^{-\frac{4}{3}}K_{4}^{-2}K_{5}^{-\frac{5}{3}}K_{6}^{-\frac{4}{3}}c^{-1}E_{3}
=K_{7}E_{3},
$$
$$
E_{4}K_{7}=E_{4}K_{1}^{-\frac{2}{3}}K_{2}^{-1}K_{3}^{-\frac{4}{3}}K_{4}^{-2}K_{5}^{-\frac{5}{3}}K_{6}^{-\frac{4}{3}}c^{-1}
=qq^{\frac{4}{3}}q^{-4}q^{\frac{5}{3}}K_{1}^{-\frac{2}{3}}K_{2}^{-1}K_{3}^{-\frac{4}{3}}K_{4}^{-2}K_{5}^{-\frac{5}{3}}K_{6}^{-\frac{4}{3}}c^{-1}E_{4}
=K_{7}E_{4},
$$
$$
E_{5}K_{7}=E_{5}K_{1}^{-\frac{2}{3}}K_{2}^{-1}K_{3}^{-\frac{4}{3}}K_{4}^{-2}K_{5}^{-\frac{5}{3}}K_{6}^{-\frac{4}{3}}c^{-1}
=q^{2}q^{-\frac{10}{3}}q^{\frac{4}{3}}K_{1}^{-\frac{2}{3}}K_{2}^{-1}K_{3}^{-\frac{4}{3}}K_{4}^{-2}K_{5}^{-\frac{5}{3}}K_{6}^{-\frac{4}{3}}c^{-1}E_{5}
=K_{7}E_{5},
$$
$$
E_{6}K_{7}=E_{6}K_{1}^{-\frac{2}{3}}K_{2}^{-1}K_{3}^{-\frac{4}{3}}K_{4}^{-2}K_{5}^{-\frac{5}{3}}K_{6}^{-\frac{4}{3}}c^{-1}
=q^{\frac{5}{3}}q^{-\frac{8}{3}}K_{1}^{-\frac{2}{3}}K_{2}^{-1}K_{3}^{-\frac{4}{3}}K_{4}^{-2}K_{5}^{-\frac{5}{3}}K_{6}^{-\frac{4}{3}}c^{-1}E_{5}
=q^{-1}K_{7}E_{6}.
$$

We will describe the $q$-Serre relations between $e^{27}$ and $E_{i}$.
Since $E_{i}$'s appear in the minor entries $(m^{+})^{i}_{i+1}$ of matrix $m^{+}$,
we have

$
\left.
\begin{array}{l}
(m^{+})^{i}_{i+1}=-(q-q^{-1})E_{j}(m^{+})^{i+1}_{i+1},\\
e^{27}(m^{+})^{i}_{i+1}=\lambda R^{i}_{i}{}^{27}_{27}(m^{+})^{i}_{i+1}e^{27},i=1,\cdots,4,\\
e^{27}(m^{+})^{i+1}_{i+1}=\lambda R^{i+1}_{i+1}{}^{27}_{27}(m^{+})^{i+1}_{i+1}e^{27} ,
\end{array}
\right\}
\Longrightarrow~
E_{7}E_{j}=E_{j}E_{7},j=1,2,3,4.
$

$
\left.
\begin{array}{l}
(m^{+})^{4}_{6}=-(q-q^{-1})E_{5}(m^{+})^{6}_{6},\\
e^{27}(m^{+})^{4}_{6}=\lambda R^{4}_{a}{}^{27}_{b}(m^{+})^{a}_{6}e^{b},\\
e^{27}(m^{+})^{6}_{6}=\lambda R^{6}_{6}{}^{27}_{27}(m^{+})^{6}_{6}e^{27},\\
(m^{+})^{5}_{6}=0,
R^{4}_{6}{}^{27}_{25}=0,
R^{4}_{4}{}^{27}_{27}=R^{6}_{6}{}^{27}_{27}=1,
\end{array}
\right\}
\Longrightarrow~
E_{7}E_{5}=E_{5}E_{7}.
$

$
\left.
\begin{array}{l}
(m^{+})^{6}_{7}=-(q-q^{-1})E_{6}(m^{+})^{7}_{7},\\
e^{27}(m^{+})^{6}_{7}=\lambda (m^{+})^{6}_{7}e^{27}+\lambda (q-q^{-1})(m^{+})^{7}_{7}e^{26},\\
e^{27}(m^{+})^{7}_{7}=\lambda R^{7}_{7}{}^{27}_{27}(m^{+})^{7}_{7}e^{27}=\lambda q(m^{+})^{7}_{7}e^{27},\\
e^{26}(m^{+})^{7}_{7}=\lambda R^{7}_{7}{}^{26}_{26}(m^{+})^{7}_{7}e^{27}=\lambda (m^{+})^{7}_{7}e^{26},
\end{array}
\right\}
\Longrightarrow
e^{26}=q^{-1}E_{6}e^{27}-e^{27}E_{6}.
$

According to
$
R^{\prime}{}^{27,}_{27,}{}^{26}_{26}=-2q,
$
and
$
R^{\prime}{}^{27,}_{26,}{}^{26}_{27}=2q^{2}+1,
$
then
$
e^{26}e^{27}=R^{\prime}{}^{27,}_{a,}{}^{26}_{b}e^{a}e^{b}
=-2qe^{27}e^{26}+(2q^{2}+1)e^{26}e^{27}
~\Longrightarrow~e^{27}e^{26}=qe^{26}e^{27}.
$
Combing with the above relation $e^{26}=q^{-1}E_{6}e^{27}-e^{27}E_{6},$
we obtain
$$
(E_{7})^{2}E_{6}-(q+q^{-1})E_{7}E_{6}E_{7}+E_{6}(E_{7})^{2}=0.
$$

On the other hand,
we need to know the relation between $e^{26}$ and $E_{6}$,
which can be explored by the following relations

$
\left\{
\begin{array}{l}
(m^{+})^{6}_{7}=(q-q^{-1})E_{6}(m^{+})^{7}_{7},\\
e^{26}(m^{+})^{6}_{7}=\lambda R^{6}_{6}{}^{26}_{26}(m^{+})^{6}_{7}e^{26}=\lambda q(m^{+})^{6}_{7}e^{26},\\
e^{26}(m^{+})^{7}_{7}=\lambda R^{7}_{7}{}^{26}_{26}(m^{+})^{7}_{7}e^{26}=\lambda (m^{+})^{7}_{7}e^{26}.
\end{array}
\right.
\Longrightarrow
 e^{26}E_{6}=qE_{6}e^{26}.
$

Combing with the above relation $e^{26}=q^{-1}E_{6}e^{27}-e^{27}E_{6}$ again,
we obtain
$$
(E_{6})^{2}E_{7}-(q+q^{-1})E_{6}E_{7}E_{6}+E_{7}(E_{6})^{2}=0.
$$
The relations between $f_{27}$ and others $F_{i},K_{i}$'s can be obtained in a similar way.
With these relations,
we prove that the new quantum group is $U_{q}(E_{7})$.
\end{proof}

\subsection{Inductive construction of $U_{q}(E_{8})$}
We choose the $56$-dimensional minimal fundamental representation $T_{V}$ of $U_{q}(E_{7})$ to construct $U_{q}(E_{8})$.
The representation is given by the following Figure $3$,
which also appeared in \cite{KKM}.
This representation also satisfies
$E_{\beta}(f_{m})=f_{n} \Longleftrightarrow F_{\beta}(f_{n})=f_{m}$
and
$E_{i}^{2}$ is zero action.
Moreover,
the data $R,R^{\prime}$,
and the entries we need in $m^{\pm}$
are given by the following Lemma.
\begin{lemma}
$(1)$ The minimal polynomial of the braiding $PR_{VV}$ is
$$(P{R}_{VV}+q^{-\frac{1}{2}}I)(P{R}_{VV}+q^{\frac{1}{2}}I)(P{R}_{VV}-q^{\frac{1}{2}}I)(P{R}_{VV}-q^{\frac{3}{2}}I)=0.$$
Thus setting $R=q^{\frac{1}{2}}R_{VV},R^{\prime}=-q^{-4}R(P{R})^{2}+q^{-2}RP{R}+q^{-2}R$,
we have $(PR+I)(PR^{\prime}-I)=0$.

$(2)$  The entries we need in $m^{\pm}$ are listed as follows
\begin{itemize}
  \item In $m^{+}$:

$
\left.\begin{array}{ll}
(m^{+})^{1}_{2}=-(q-q^{-1})E_{7}K_{1}K_{2}^{\frac{3}{2}}K_{3}^{2}K_{4}^{3}K_{5}^{\frac{5}{2}}K_{6}^{2}K_{7}^{\frac{1}{2}},&
(m^{+})^{2}_{2}=K_{1}K_{2}^{\frac{3}{2}}K_{3}^{2}K_{4}^{3}K_{5}^{\frac{5}{2}}K_{6}^{2}K_{7}^{\frac{1}{2}},
\\
(m^{+})^{2}_{3}=-(q-q^{-1})E_{6}K_{1}K_{2}^{\frac{3}{2}}K_{3}^{2}K_{4}^{3}K_{5}^{\frac{5}{2}}K_{6}K_{7}^{\frac{1}{2}},&
(m^{+})^{3}_{3}=K_{1}K_{2}^{\frac{3}{2}}K_{3}^{2}K_{4}^{3}K_{5}^{\frac{5}{2}}K_{6}K_{7}^{\frac{1}{2}},
\\
(m^{+})^{3}_{4}=-(q-q^{-1})E_{5}K_{1}K_{2}^{\frac{3}{2}}K_{3}^{2}K_{4}^{3}K_{5}^{\frac{3}{2}}K_{6}K_{7}^{\frac{1}{2}},&
(m^{+})^{4}_{4}=K_{1}K_{2}^{\frac{3}{2}}K_{3}^{2}K_{4}^{3}K_{5}^{\frac{3}{2}}K_{6}K_{7}^{\frac{1}{2}},
\\
(m^{+})^{4}_{5}=-(q-q^{-1})E_{4}K_{1}K_{2}^{\frac{3}{2}}K_{3}^{2}K_{4}^{2}K_{5}^{\frac{3}{2}}K_{6}K_{7}^{\frac{1}{2}},&
(m^{+})^{5}_{5}=K_{1}K_{2}^{\frac{3}{2}}K_{3}^{2}K_{4}^{2}K_{5}^{\frac{3}{2}}K_{6}K_{7}^{\frac{1}{2}},
\\
(m^{+})^{5}_{6}=-(q-q^{-1})E_{3}K_{1}K_{2}^{\frac{3}{2}}K_{3}K_{4}^{2}K_{5}^{\frac{3}{2}}K_{6}K_{7}^{\frac{1}{2}},&
(m^{+})^{6}_{6}=K_{1}K_{2}^{\frac{3}{2}}K_{3}K_{4}^{2}K_{5}^{\frac{3}{2}}K_{6}K_{7}^{\frac{1}{2}},
\\
(m^{+})^{5}_{7}{=}(q^{-1}-q)E_{2}K_{1}K_{2}^{\frac{1}{2}}K_{3}^{2}K_{4}^{2}K_{5}^{\frac{3}{2}}K_{6}K_{7}^{\frac{1}{2}},&
(m^{+})^{7}_{7}{=}K_{1}K_{2}^{\frac{1}{2}}K_{3}^{2}K_{4}^{2}K_{5}^{\frac{3}{2}}K_{6}K_{7}^{\frac{1}{2}},
(m^{+})^{6}_{7}{=}0,
\\
(m^{+})^{6}_{8}=-(q-q^{-1})E_{1}K_{2}^{\frac{3}{2}}K_{3}K_{4}^{2}K_{5}^{\frac{3}{2}}K_{6}K_{7}^{\frac{1}{2}},&
(m^{+})^{8}_{8}=K_{2}^{\frac{3}{2}}K_{3}K_{4}^{2}K_{5}^{\frac{3}{2}}K_{6}K_{7}^{\frac{1}{2}},~~
(m^{+})^{7}_{8}=0,
\\
(m^{+})^{1}_{1}=K_{1}K_{2}^{\frac{3}{2}}K_{3}^{2}K_{4}^{3}K_{5}^{\frac{5}{2}}K_{6}^{2}K_{7}^{\frac{3}{2}},&
(m^{+})^{56}_{56}=K_{1}^{-1}K_{2}^{-\frac{3}{2}}K_{3}^{-2}K_{4}^{-3}K_{5}^{-\frac{5}{2}}K_{6}^{-2}K_{7}^{-\frac{3}{2}};
\end{array}
\right.
$

\item In $m^{-}$:

$
\left.\begin{array}{ll}
(m^{-})^{2}_{1}=\frac{K_{1}^{-1}K_{2}^{-\frac{3}{2}}K_{3}^{-2}K_{4}^{-3}K_{5}^{-\frac{5}{2}}K_{6}^{-2}K_{7}^{-\frac{1}{2}}F_{7}}{q^{2}-1},&
(m^{-})^{2}_{2}=K_{1}^{-1}K_{2}^{-\frac{3}{2}}K_{3}^{-2}K_{4}^{-3}K_{5}^{-\frac{5}{2}}K_{6}^{-2}K_{7}^{-\frac{1}{2}},
\\
(m^{-})^{3}_{2}=\frac{K_{1}^{-1}K_{2}^{-\frac{3}{2}}K_{3}^{-2}K_{4}^{-3}K_{5}^{-\frac{5}{2}}K_{6}^{-1}K_{7}^{-\frac{1}{2}}F_{6}}{q^{2}-1},&
(m^{-})^{3}_{3}=K_{1}^{-1}K_{2}^{-\frac{3}{2}}K_{3}^{-2}K_{4}^{-3}K_{5}^{-\frac{5}{2}}K_{6}^{-1}K_{7}^{-\frac{1}{2}},
\\
(m^{-})^{4}_{3}=\frac{K_{1}^{-1}K_{2}^{-\frac{3}{2}}K_{3}^{-2}K_{4}^{-3}K_{5}^{-\frac{3}{2}}K_{6}^{-1}K_{7}^{-\frac{1}{2}}F_{5}}{q^{2}-1},&
(m^{-})^{4}_{4}=K_{1}^{-1}K_{2}^{-\frac{3}{2}}K_{3}^{-2}K_{4}^{-3}K_{5}^{-\frac{3}{2}}K_{6}^{-1}K_{7}^{-\frac{1}{2}},
\\
(m^{-})^{5}_{4}=\frac{K_{1}^{-1}K_{2}^{-\frac{3}{2}}K_{3}^{-2}K_{4}^{-2}K_{5}^{-\frac{3}{2}}K_{6}^{-1}K_{7}^{-\frac{1}{2}}F_{4}}{q^{2}-1},&
(m^{-})^{5}_{5}=K_{1}^{-1}K_{2}^{-\frac{3}{2}}K_{3}^{-2}K_{4}^{-2}K_{5}^{-\frac{3}{2}}K_{6}^{-1}K_{7}^{-\frac{1}{2}},
\\
(m^{-})^{6}_{5}=\frac{K_{1}^{-1}K_{2}^{-\frac{3}{2}}K_{3}^{-1}K_{4}^{-2}K_{5}^{-\frac{3}{2}}K_{6}^{-1}K_{7}^{-\frac{1}{2}}F_{3}}{q^{2}-1},&
(m^{-})^{6}_{6}=K_{1}^{-1}K_{2}^{-\frac{3}{2}}K_{3}^{-1}K_{4}^{-2}K_{5}^{-\frac{3}{2}}K_{6}^{-1}K_{7}^{-\frac{1}{2}},
\\
(m^{-})^{8}_{6}=\frac{K_{2}^{-\frac{3}{2}}K_{3}^{-1}K_{4}^{-2}K_{5}^{-\frac{3}{2}}K_{6}^{-1}K_{7}^{-\frac{1}{2}}F_{1}}{q^{2}-1},&
(m^{-})^{8}_{8}=K_{2}^{-\frac{3}{2}}K_{3}^{-1}K_{4}^{-2}K_{5}^{-\frac{3}{2}}K_{6}^{-1}K_{7}^{-\frac{1}{2}},
(m^{-})^{7}_{6}=0,
\\
(m^{-})^{10}_{8}=\frac{K_{2}^{-\frac{1}{2}}K_{3}^{-1}K_{4}^{-2}K_{5}^{-\frac{3}{2}}K_{6}^{-1}K_{7}^{-\frac{1}{2}}F_{2}}{q^{2}-1},&
(m^{-})^{10}_{10}=K_{2}^{-\frac{1}{2}}K_{3}^{-1}K_{4}^{-2}K_{5}^{-\frac{3}{2}}K_{6}^{-1}K_{7}^{-\frac{1}{2}},
(m^{-})^{9}_{8}=0,
\\
(m^{-})^{1}_{1}=K_{1}^{-1}K_{2}^{-\frac{3}{2}}K_{3}^{-2}K_{4}^{-3}K_{5}^{-\frac{5}{2}}K_{6}^{-2}K_{7}^{-\frac{3}{2}},&
(m^{-})^{56}_{56}=K_{1}K_{2}^{\frac{3}{2}}K_{3}^{2}K_{4}^{3}K_{5}^{\frac{5}{2}}K_{6}^{2}K_{7}^{\frac{3}{2}}.
\end{array}
\right.$
\end{itemize}
\end{lemma}
\begin{proof} The proof is similar.
The decomposition of this $56$-dimensional representation is
$$
(0000001)\otimes(0000001)=(0000002)\oplus(0000010)\oplus(1000000)\oplus(0000000),
$$
namely,
$
V\otimes V=V_{1}\oplus V_{2}\oplus V_{3}\oplus V_{4},
$
where
$V_{i},i=1,2,3,4$ is the irreducible representation wit highest weight
$2\lambda_{7},\lambda_{6},\lambda_{1},0$ respectively.
This means matrix $PR_{VV}$ only has $4$ distinct eigenvalues,
denoted by
$x_{i},i=1,2,3,4$.
We set
$
\triangle_{1}=x_{1}+x_{2}+x_{3}+x_{4},
\triangle_{2}=x_{1}x_{2}+x_{1}x_{3}+x_{2}x_{3}+(x_{1}+x_{2}+x_{3})x_{4},
\triangle_{3}=x_{1}x_{2}x_{3}+x_{1}x_{2}x_{4}+x_{1}x_{3}x_{4}+x_{2}x_{3}x_{4},
\triangle_{4}=x_{1}x_{2}x_{3}x_{4},
$
and $\mathcal{N}=(PR_{VV}-x_{1}I)(PR_{VV}-x_{2}I)(PR_{VV}-x_{3}I)$ $(PR_{VV}-x_{4}I)$.

We find nonzero entries at rows $(12)$ and $(21)$ in matrix $PR_{VV}-x_{i}I$ are
$
(PR_{VV}-x_{i}I)^{1}_{1}{}^{2}_{2}=-x_{i},
(PR_{VV}-x_{i}I)^{1}_{2}{}^{2}_{1}=q^{\frac{1}{2}},
(PR_{VV}-x_{i}I)^{2}_{1}{}^{1}_{2}=q^{\frac{1}{2}},
(PR_{VV}-x_{i}I)^{2}_{2}{}^{1}_{1}=q^{\frac{1}{2}}(q-q^{-1})-x_{i}.
$
So we obtain the entries at row $(1,2)$ and columns $(1,2)$ and $(2,1)$ in matrix $\mathcal{N}$ are

$
\mathcal{N}^{1}_{1}{}^{2}_{2}=\Delta_{4}+q\Delta_{2}
-q^{\frac{3}{2}}(q-q^{-1})\Delta_{1}+q^{2}+q^{2}(q-q^{-1})^{2},
$

$
\mathcal{N}^{1}_{2}{}^{2}_{1}=-\Delta_{3}
+q^{\frac{1}{2}}(q-q^{-1})\Delta_{2}
-[q+q(q-q^{-1})^{2}]\Delta_{1}
+q^{\frac{3}{2}}(q-q^{-1})^{3}+2q^{\frac{3}{2}}(q-q^{-1}).
$

We will consider the entries at rows $(5,11)$ and $(11,5)$.
From the Figure $3$,
we find that the basis $f_{5}$ can be raised by the roots $\alpha_{2}$,$\alpha_{3}$ and $\alpha_{1}+\alpha_{3}$,
but the basis $f_{11}$ is lowered by the roots $\alpha_{4}$,$\alpha_{3}+\alpha_{4}$ and $\alpha_{2}+\alpha_{4}$.
This fact means that raising $f_{5}$ and lowering $f_{11}$ can't be done by the same roots,
so the only nonzero entry at row $(5,11)$ in matrix $R_{VV}$ is
$
R_{VV}{}^{5,}_{5,}{}^{11}_{11}=q^{\frac{1}{2}}.
$
We also obtain that the only nonzero entry at row $(11,5)$ in matrix $R_{VV}$ is
$
R_{VV}{}^{11,}_{11,}{}^{5}_{5}=q^{\frac{1}{2}}
$
by the same analysis.
Then the nonzero entries at rows $(5,11)$ and $(11,5)$ in matrix $PR_{VV}-x_{i}I$ are
$
(PR_{VV}-x_{i}I)^{5,11}_{5,11}=-x_{i},
(PR_{VV}-x_{i}I)^{5,11}_{11,5}=q^{\frac{1}{2}},
(PR_{VV}-x_{i}I)^{11,5}_{5,11}=q^{\frac{1}{2}},
(PR_{VV}-x_{i}I)^{11,5}_{11,5}=-x_{i}.
$
Then the entries at row $(5,11)$ and columns $(5,11)$ and $(11,5)$ in matrix $\mathcal{N}$ are
$
\mathcal{N}^{5,11}_{5,11}=\Delta_{4}+q\Delta_{2}+q^{2},
\mathcal{N}^{5,11}_{11,5}=-q^{\frac{1}{2}}\Delta_{3}-q^{\frac{3}{2}}\Delta_{1}.
$

These entries yield the following equations

$
\left\{
\begin{array}{l}
\Delta_{4}+q\Delta_{2}
-q^{\frac{3}{2}}(q-q^{-1})\Delta_{1}+q^{2}+q^{2}(q-q^{-1})^{2}=0,\\
-\Delta_{3}
+q^{\frac{1}{2}}(q-q^{-1})\Delta_{2}
-[q+q(q-q^{-1})^{2}]\Delta_{1}
+q^{\frac{3}{2}}(q-q^{-1})^{3}+2q^{\frac{3}{2}}(q-q^{-1})=0,\\
\Delta_{4}+q\Delta_{2}+q^{2}=0,\\
-q^{\frac{1}{2}}\Delta_{3}
-q^{\frac{3}{2}}\Delta_{1}=0.
\end{array}
\right.
$

Solving it, we get these eigenvalues are
$
-q^{-\frac{1}{2}},
\pm q^{\frac{1}{2}},
q^{\frac{3}{2}}.
$
Then the minimal polynomial of $PR_{VV}$ is
$(P{R}_{VV}+q^{-\frac{1}{2}}I)(P{R}_{VV}+q^{\frac{1}{2}}I)(P{R}_{VV}-q^{\frac{1}{2}}I)(P{R}_{VV}-q^{\frac{3}{2}}I)=0.$
\end{proof}

With these,
we have the following
\begin{theorem}  With the quantum normalization constant
$\lambda=q^{-\frac{1}{2}}$,
identify $e^{56}, f_{56}$ and $(m^{+})^{56}_{56}c^{-1}$ with the additional simple root vectors $E_{8}, F_{8}$, and the group-element $K_{8}$,
then the new quantum group
$U(V^{\vee}(R^{\prime},R_{21}^{-1}),\widetilde{U_{q}^{ext}(E_{7})},V(R^{\prime},R))=U_{q}(E_{8})$
with $K_{i}^{\frac{1}{2}}$ adjoined.
\end{theorem}
\begin{proof}
The cross relation can be obtained easily in view of the relations in Theorem \ref{cor1},
so we only focus on the $q$-Serre relations between $e^{56}$ and $E_{i}$.
$E_{1},\, E_{2}$ appear in $(m^{+})^{6}_{8},\, (m^{+})^{5}_{7}$, respectively,
then according to Theorem \ref{cor1},
we obtain

$
\left.
\begin{array}{l}
(m^{+})^{6}_{8}=-(q-q^{-1})E_{1}(m^{+})^{8}_{8},\\
e^{56}(m^{+})^{6}_{8}=\lambda R^{6}_{6}{}^{56}_{56}(m^{+})^{6}_{8}e^{56}=\lambda (m^{+})^{6}_{8}e^{56},\\
e^{56}(m^{+})^{8}_{8}
=\lambda R^{8}_{8}{}^{56}_{56}(m^{+})^{6}_{8}e^{56}
=\lambda (m^{+})^{8}_{8}e^{56},
\end{array}
\right\}
\Longrightarrow
E_{8}E_{1}=E_{1}E_{8}.
$

$
\left.
\begin{array}{l}
(m^{+})^{5}_{7}=-(q-q^{-1})E_{2}(m^{+})^{7}_{7},\\
e^{56}(m^{+})^{5}_{7}=\lambda R^{5}_{5}{}^{56}_{56}(m^{+})^{5}_{7}e^{56}=\lambda (m^{+})^{5}_{7}e^{56},\\
e^{56}(m^{+})^{7}_{7}=\lambda R^{7}_{7}{}^{56}_{56}(m^{+})^{7}_{7}e^{56}=\lambda (m^{+})^{7}_{7}e^{56},
\end{array}
\right\}
\Longrightarrow
E_{8}E_{2}=E_{2}E_{8}.
$

We observed that other $E_{i}$'s locate in the minor entries $(m^{+})^{i}_{i+1}$ from the above Lemma,
then the following relations can be deduced from Theorem \ref{cor1}.

$
\left.
\begin{array}{l}
e^{56}(m^{+})^{1}_{2}=\lambda R^{1}_{1}{}^{56}_{56}(m^{+})^{1}_{2}e^{56}+\lambda R^{1}_{2}{}^{56}_{55}(m^{+})^{2}_{2}e^{55},\\
e^{56}(m^{+})^{i}_{i+1}=\lambda R^{i}_{i}{}^{56}_{56}(m^{+})^{i}_{i+1}e^{56}, 2\leq i\leq 5,\\
(m^{+})^{i}_{i+1}=-(q-q^{-1})E_{j}(m^{+})^{i+1}_{i+1},i+j=8,\\
e^{56}(m^{+})^{k}_{k}=\lambda R^{k}_{k}{}^{56}_{56}(m^{+})^{k}_{k}e^{56}, \forall k,\\
R^{2}_{2}{}^{55}_{55}=q^{-1},
R^{1,}_{2,}{}^{56}_{55}=q^{-1}(q-q^{-1}),\\
R^{1}_{1}{}^{56}_{56}=q^{-1},
R^{i}_{i}{}^{56}_{56}=1, 2\leq i\leq 8,
\end{array}
\right\}
\Longrightarrow
\left\{
\begin{array}{l}
E_{8}E_{j}=E_{j}E_{8},~3\leq j\leq 6,\\
e^{55}=qE_{7}e^{56}-e^{56}E_{7}.
\end{array}
\right.
$

According to
$
R^{\prime}{}^{56,}_{56,}{}^{55}_{55}=q^{-1}-q^{-3},
$
and
$
R^{\prime}{}^{56,}_{55,}{}^{55}_{56}=q^{-2},
$
then
$e^{55}e^{56}=R^{\prime}{}^{56}_{a}{}^{55}_{b}e^{a}e^{b}=(q^{-1}-q^{-3})e^{56}e^{55}+q^{-2}e^{55}e^{56}~\Longrightarrow~e^{55}e^{56}=q^{-1}e^{56}e^{55}$.
Combining with
$e^{55}=qE_{7}e^{56}-e^{56}E_{7}$,
we obtain
$$
(E_{8})^{2}E_{7}-(q+q^{-1})E_{8}E_{7}E_{8}+E_{7}(E_{8})^{2}=0.
$$
On the other hand,
we can get the relations between $e^{55}$ and $E_{7}$ by the following relations.

$
\left.
\begin{array}{l}
m^{+}{}^{1}_{2}=-(q-q^{-1})E_{7}m^{+}{}^{2}_{2},\\
e^{55}m^{+}{}^{1}_{2}=\lambda R^{1,}_{1,}{}^{55}_{55}m^{+}{}^{1}_{2}e^{55},\\
e^{55}m^{+}{}^{2}_{2}=\lambda R^{2,}_{2,}{}^{55}_{55}m^{+}{}^{2}_{2}e^{55},
\end{array}
\right\}
\Longrightarrow
e^{55}E_{7}=q^{-1}E_{7}e^{55}.
$

Combining with $e^{55}=qE_{7}e^{56}-e^{56}E_{7}$ again,
we obtain
$$
(E_{7})^{2}E_{8}-(q+q^{-1})E_{7}E_{8}E_{7}+E_{8}(E_{7})^{2}=0.
$$

With these relations,
we prove that the new quantum group is $U_{q}(E_{8})$.

This completes the proof.
\end{proof}
\newpage
\setlength{\unitlength}{1mm}
\begin{picture}(60,168)
\put(0,166)
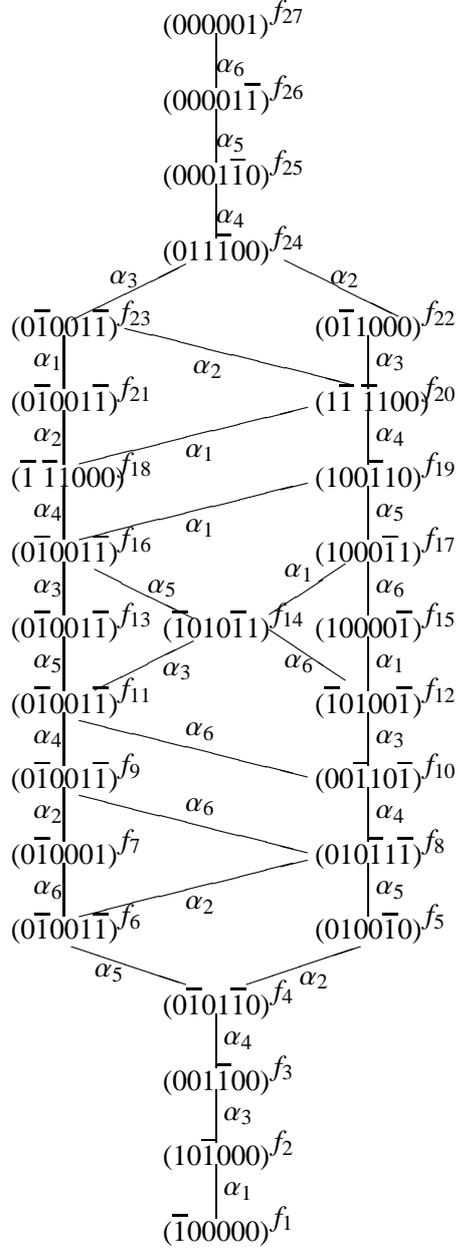
{Figure $2$.\quad $U_{q}(E_{6})$'s $27$-dimensional minimal fundamental representation}
\put(20,0){$(\overline{1}00000)$}
\put(34,2){$f_{1}$}
\put(27,3){\line(0,1){7}}
\put(28,6){$\alpha_{1}$}

\put(20,10){$(10\overline{1}000)$}
\put(34,12){$f_{2}$}
\put(27,13){\line(0,1){7}}
\put(28,16){$\alpha_{3}$}

\put(20,20){$(001\overline{1}00)$}
\put(34,22){$f_{3}$}
\put(27,23){\line(0,1){7}}
\put(28,26){$\alpha_{4}$}

\put(20,30){$(0\overline{1}01\overline{1}0)$}
\put(34,32){$f_{4}$}
\put(8,39){\line(3,-1){15}}
\put(11,35){$\alpha_{5}$}
\put(31,34){\line(3,1){15}}
\put(38,34){$\alpha_{2}$}

\put(40,40){$(0100\overline{1}0)$}
\put(54,42){$f_{5}$}
\put(47,43){\line(0,1){7}}
\put(48,46){$\alpha_{5}$}

\put(40,50){$(010\overline{1}1\overline{1})$}
\put(54,52){$f_{8}$}
\put(47,53){\line(0,1){7}}
\put(48,56){$\alpha_{4}$}

\put(40,60){$(00\overline{1}10\overline{1})$}
\put(54,62){$f_{10}$}
\put(47,63){\line(0,1){7}}
\put(48,66){$\alpha_{3}$}

\put(40,70){$(\overline{1}0100\overline{1})$}
\put(54,72){$f_{12}$}
\put(47,73){\line(0,1){7}}
\put(48,76){$\alpha_{1}$}

\put(40,80){$(10000\overline{1})$}
\put(54,82){$f_{15}$}
\put(47,83){\line(0,1){7}}
\put(48,86){$\alpha_{6}$}

\put(40,90){$(1000\overline{1}1)$}
\put(54,92){$f_{17}$}
\put(47,93){\line(0,1){7}}
\put(48,96){$\alpha_{5}$}

\put(40,100){$(100\overline{1}10)$}
\put(54,102){$f_{19}$}
\put(47,103){\line(0,1){7}}
\put(48,106){$\alpha_{4}$}

\put(40,110){$(1\overline{1}~\overline{1}100)$}
\put(54,112){$f_{20}$}
\put(47,113){\line(0,1){7}}
\put(48,116){$\alpha_{3}$}

\put(40,120){$(0\overline{1}1000)$}
\put(54,122){$f_{22}$}

\put(0,40){$(0\overline{1}001\overline{1})$}
\put(14,42){$f_{6}$}
\put(7,43){\line(0,1){7}}
\put(3,46){$\alpha_{6}$}
\put(9,43){\line(4,1){30}}
\put(23,44){$\alpha_{2}$}

\put(0,50){$(0\overline{1}0001)$}
\put(14,52){$f_{7}$}
\put(7,53){\line(0,1){7}}
\put(3,56){$\alpha_{2}$}

\put(0,60){$(0\overline{1}001\overline{1})$}
\put(14,62){$f_{9}$}
\put(7,63){\line(0,1){7}}
\put(3,66){$\alpha_{4}$}
\put(9,59){\line(4,-1){30}}
\put(23,57){$\alpha_{6}$}

\put(0,70){$(0\overline{1}001\overline{1})$}
\put(14,72){$f_{11}$}
\put(7,73){\line(0,1){7}}
\put(3,76){$\alpha_{5}$}
\put(9,69){\line(4,-1){30}}
\put(23,67){$\alpha_{6}$}
\put(11,73){\line(2,1){13}}
\put(20,75){$\alpha_{3}$}

\put(0,80){$(0\overline{1}001\overline{1})$}
\put(14,82){$f_{13}$}
\put(7,83){\line(0,1){7}}
\put(3,86){$\alpha_{3}$}

\put(0,90){$(0\overline{1}001\overline{1})$}
\put(14,92){$f_{16}$}
\put(7,93){\line(0,1){7}}
\put(3,96){$\alpha_{4}$}
\put(11,89){\line(2,-1){13}}
\put(18,86){$\alpha_{5}$}
\put(9,93){\line(4,1){30}}
\put(23,94){$\alpha_{1}$}

\put(0,100){$(\overline{1}~\overline{1}1000)$}
\put(14,102){$f_{18}$}
\put(7,103){\line(0,1){7}}
\put(3,106){$\alpha_{2}$}
\put(9,103){\line(4,1){30}}
\put(23,104){$\alpha_{1}$}

\put(0,110){$(0\overline{1}001\overline{1})$}
\put(14,112){$f_{21}$}
\put(7,113){\line(0,1){7}}
\put(3,116){$\alpha_{1}$}

\put(0,120){$(0\overline{1}001\overline{1})$}
\put(14,122){$f_{23}$}
\put(15,121){\line(4,-1){30}}
\put(24,115){$\alpha_{2}$}
\put(8,122){\line(2,1){15}}
\put(13,127){$\alpha_{3}$}

\put(20,80){$(\overline{1}010\overline{1}1)$}
\put(34,82){$f_{14}$}
\put(34,83){\line(3,2){10}}
\put(36,88){$\alpha_{1}$}
\put(34,81){\line(3,-2){10}}
\put(36,76){$\alpha_{6}$}

\put(20,130){$(011\overline{1}00)$}
\put(34,132){$f_{24}$}
\put(36,130){\line(2,-1){15}}
\put(42,127){$\alpha_{2}$}
\put(27,133){\line(0,1){7}}
\put(27,135){$\alpha_{4}$}

\put(20,140){$(0001\overline{1}0)$}
\put(34,142){$f_{25}$}
\put(27,143){\line(0,1){7}}
\put(27,145){$\alpha_{5}$}

\put(20,150){$(00001\overline{1})$}
\put(34,152){$f_{26}$}
\put(27,153){\line(0,1){7}}
\put(27,155){$\alpha_{6}$}

\put(20,160){$(000001)$}
\put(34,162){$f_{27}$}
\end{picture}
\newpage
\setlength{\unitlength}{1mm}
\begin{picture}(90,200)
\put(0,197){Figure $3.$\quad $U_{q}(E_{7})$'s $56$-dimensional minimal fundamental representation}
\put(30,0){$(000000\overline{1})$}
\put(46,2){$f_{1}$}
\put(38,2.5){\line(0,1){4.5}}
\put(38.5,4){$\alpha_{7}$}
\put(30,7){($00000\overline{1}1)$}
\put(46,9){$f_{2}$}
\put(38,9){\line(0,1){5}}
\put(38.5,11){$\alpha_{6}$}
\put(30,14){$(0000\overline{1}10)$}
\put(46,16){$f_{3}$}
\put(38,16){\line(0,1){5}}
\put(38.5,18){$\alpha_{5}$}
\put(30,21){($000\overline{1}100)$}
\put(46,23){$f_{4}$}
\put(38,24){\line(0,1){4}}
\put(38.5,25){$\alpha_{4}$}
\put(30,28){$(0\overline{1}~\overline{1}1000)$}
\put(47,30){$f_{5}$}
\put(25,34){\line(3,-2){5.5}}
\put(27,32){$\alpha_{3}$}
\put(43,31){\line(3,2){5}}
\put(48,32.5){$\alpha_{2}$}
\put(15,35){$(\overline{1}~\overline{1}10000)$}
\put(31,37){$f_{6}$}
\put(33,38){\line(3,1){12}}
\put(37,38){$\alpha_{2}$}
\put(25,37){\line(0,1){5}}
\put(21,39){$\alpha_{1}$}
\put(45,35){$(01\overline{1}000)$}
\put(59,37){$f_{7}$}
\put(55,37){\line(0,1){5}}
\put(51,39){$\alpha_{3}$}
\put(15,42){$(1\overline{1}00000)$}
\put(31,44){$f_{8}$}
\put(25,44){\line(0,1){5}}
\put(21,46){$\alpha_{2}$}
\put(45,42){$(\overline{1}11\overline{1}000)$}
\put(59,44){$f_{9}$}
\put(55,44){\line(0,1){5}}
\put(51,46){$\alpha_{4}$}
\put(32,51){\line(3,-2){13}}
\put(40,46){$\alpha_{1}$}

\put(15,49){$(110\overline{1}000)$}
\put(31,51){$f_{10}$}
\put(25,51){\line(0,1){5}}
\put(21,54){$\alpha_{4}$}
\put(45,49){$(\overline{1}001\overline{1}00)$}
\put(59,51){$f_{11}$}
\put(55,51){\line(0,1){5}}
\put(51,54){$\alpha_{5}$}
\put(32,58){\line(3,-2){13}}
\put(40,53){$\alpha_{1}$}

\put(15,56){$(1001\overline{1}0\overline{1})$}
\put(31,58){$f_{12}$}
\put(8,62.5){\line(3,-2){7}}
\put(7,59){$\alpha_{3}$}
\put(22,58){\line(3,1){11}}
\put(21,60){$\alpha_{5}$}

\put(45,56){$(\overline{1}0001\overline{1}0)$}
\put(59,58){$f_{13}$}
\put(43,63){\line(3,-2){7}}
\put(41,59){$\alpha_{1}$}
\put(52,58){\line(2,1){9}}
\put(52,61){$\alpha_{6}$}

\put(0,63){$(00\overline{1}0\overline{1}00)$}
\put(16,65){$f_{14}$}
\put(8,65){\line(0,1){5}}
\put(9,67){$\alpha_{5}$}

\put(0,70){$(001\overline{1}01\overline{1}0)$}
\put(16,72){$f_{17}$}
\put(8,72){\line(0,1){5}}
\put(9,74){$\alpha_{4}$}
\put(9,72){\line(4,1){21}}
\put(26,75){$\alpha_{6}$}

\put(0,77){$(0\overline{1}010\overline{1}0)$}
\put(16,79){$f_{20}$}
\put(8,79){\line(0,1){5}}
\put(9,81){$\alpha_{6}$}
\put(9,79){\line(4,1){21}}
\put(26,82){$\alpha_{2}$}

\put(0,84){$(0\overline{1}01\overline{1}1\overline{1})$}
\put(16,86){$f_{23}$}
\put(8,86){\line(0,1){5}}
\put(9,88){$\alpha_{5}$}
\put(9,86){\line(4,1){21}}
\put(26,89){$\alpha_{7}$}
\put(18,85){\line(6,1){42}}
\put(19,83){$\alpha_{2}$}

\put(0,91){$(0\overline{1}0010\overline{1})$}
\put(16,93){$f_{26}$}
\put(17,94){\line(3,1){13}}
\put(19,97){$\alpha_{2}$}

\put(0,98){$(0100\overline{1}01)$}
\put(16,100){$f_{31}$}
\put(8,100){\line(0,1){5}}
\put(4,102){$\alpha_{5}$}

\put(0,105){$(010\overline{1}1\overline{1}1)$}
\put(16,107){$f_{34}$}
\put(8,108){\line(0,1){5}}
\put(4,109){$\alpha_{6}$}
\put(10,108){\line(4,1){21}}
\put(23,109){$\alpha_{4}$}

\put(0,112){$(010\overline{1}010)$}
\put(16,114){$f_{37}$}
\put(8,115){\line(0,1){5}}
\put(4,117){$\alpha_{4}$}

\put(0,119){$(00\overline{1}1\overline{1}10)$}
\put(16,121){$f_{40}$}
\put(8,122){\line(0,1){5}}
\put(4,124){$\alpha_{5}$}
\put(10,122){\line(4,1){21}}
\put(26,124){$\alpha_{3}$}

\put(0,126){$(00\overline{1}0100)$}
\put(16,128){$f_{43}$}
\put(10,129){\line(3,1){9}}
\put(8,131){$\alpha_{3}$}

\put(30,63){$(10\overline{1}01\overline{1}0)$}
\put(46,65){$f_{15}$}
\put(30,66){\line(-3,1){13}}
\put(26,68){$\alpha_{3}$}
\put(40,65){\line(0,1){5}}
\put(41,67){$\alpha_{6}$}

\put(30,70){$(10\overline{1}001\overline{1})$}
\put(46,72){$f_{18}$}
\put(47,72){\line(3,1){13}}
\put(50,75){$\alpha_{7}$}
\put(40,72){\line(0,1){5}}
\put(41,74){$\alpha_{3}$}

\put(30,77){$(001\overline{1}01\overline{1})$}
\put(46,79){$f_{21}$}
\put(38,79){\line(-3,1){21}}
\put(35,81){$\alpha_{4}$}
\put(42,80){\line(3,1){19}}
\put(50,81){$\alpha_{7}$}

\put(30,84){$(0\overline{1}000\overline{1}0)$}
\put(46,86){$f_{24}$}
\put(46,87){\line(4,1){14}}
\put(53,87){$\alpha_{2}$}

\put(30,91){$(0\overline{1}01\overline{1}01)$}
\put(46,93){$f_{27}$}
\put(38,94){\line(-4,1){21}}
\put(34,95){$\alpha_{2}$}
\put(42,94){\line(3,1){19}}
\put(50,95){$\alpha_{5}$}

\put(30,98){$(0100\overline{1}10\overline{1})$}
\put(46,100){$f_{29}$}
\put(38,101){\line(-4,1){21}}
\put(25,104){$\alpha_{7}$}
\put(42,101){\line(3,1){19}}
\put(50,102){$\alpha_{4}$}

\put(30,105){$(0\overline{1}00010)$}
\put(46,107){$f_{33}$}
\put(38,108){\line(-4,1){21}}
\put(32,109){$\alpha_{2}$}

\put(30,112){$(00\overline{1}10\overline{1}1)$}
\put(46,114){$f_{36}$}
\put(35,115){\line(-4,1){16}}
\put(28,117){$\alpha_{6}$}
\put(40,115){\line(0,1){5}}
\put(41,117){$\alpha_{3}$}

\put(30,119){$(\overline{1}0100\overline{1}1)$}
\put(46,121){$f_{39}$}
\put(47,122){\line(3,1){13}}
\put(50,125){$\alpha_{1}$}
\put(40,122){\line(0,1){5}}
\put(41,124){$\alpha_{6}$}

\put(30,126){$(\overline{1}010\overline{1}10)$}
\put(46,128){$f_{42}$}
\put(38,129){\line(-4,1){12}}
\put(32,131){$\alpha_{5}$}
\put(42,129){\line(3,1){9}}
\put(44,131){$\alpha_{1}$}

\put(60,63){$(\overline{1}00001\overline{1})$}
\put(76,65){$f_{16}$}
\put(70,65){\line(-4,1){21}}
\put(65,67){$\alpha_{1}$}
\put(72,65){\line(0,1){5}}
\put(73,67){$\alpha_{7}$}

\put(60,70){$(\overline{1}000001)$}
\put(76,72){$f_{19}$}
\put(72,72){\line(0,1){5}}
\put(73,74){$\alpha_{1}$}

\put(60,77){$(10\overline{1}0010)$}
\put(76,79){$f_{22}$}
\put(72,80){\line(0,1){4}}
\put(73,81){$\alpha_{3}$}

\put(60,84){$(001\overline{1}001)$}
\put(76,86){$f_{25}$}
\put(70,87){\line(-4,1){21}}
\put(65,89){$\alpha_{4}$}

\put(60,91){$(0100\overline{1}1\overline{1})$}
\put(76,93){$f_{28}$}
\put(70,94){\line(-4,1){21}}
\put(65,96){$\alpha_{5}$}

\put(60,98){$(0\overline{1}001\overline{1}1)$}
\put(76,100){$f_{30}$}
\put(70,101){\line(-4,1){21}}
\put(65,103){$\alpha_{6}$}
\put(60,99){\line(-6,1){40}}
\put(40,103){$\alpha_{2}$}

\put(60,105){$(00\overline{1}100\overline{1})$}
\put(76,107){$f_{32}$}
\put(70,108){\line(-4,1){21}}
\put(65,109){$\alpha_{7}$}
\put(72,107){\line(0,1){5}}
\put(73,109){$\alpha_{3}$}

\put(60,112){$(\overline{1}01000\overline{1})$}
\put(76,114){$f_{35}$}
\put(70,115){\line(-4,1){21}}
\put(65,117){$\alpha_{7}$}
\put(72,114){\line(0,1){5}}
\put(73,116){$\alpha_{1}$}

\put(60,119){$(100000\overline{1})$}
\put(76,121){$f_{38}$}
\put(72,122){\line(0,1){5}}
\put(73,123){$\alpha_{7}$}

\put(60,126){$(10000\overline{1}1)$}
\put(76,128){$f_{41}$}
\put(70,129){\line(-4,1){15}}
\put(65,131){$\alpha_{6}$}

\put(15,133){$(\overline{1}01\overline{1}100)$}
\put(31,135){$f_{45}$}
\put(32,136){\line(3,1){13}}
\put(37,136){$\alpha_{1}$}
\put(25,135){\line(0,1){5}}
\put(20,138){$\alpha_{4}$}

\put(15,140){$(\overline{1}~\overline{1}01000)$}
\put(31,142){$f_{47}$}
\put(32,143){\line(3,1){13}}
\put(37,143){$\alpha_{1}$}
\put(25,142){\line(0,1){5}}
\put(20,145){$\alpha_{2}$}

\put(15,147){$(\overline{1}100000)$}
\put(31,149){$f_{49}$}
\put(25,149){\line(0,1){5}}
\put(20,152){$\alpha_{1}$}

\put(15,154){$(110000\overline{1})$}
\put(31,156){$f_{51}$}
\put(22,157){\line(3,1){9}}
\put(24,159){$\alpha_{3}$}

\put(45,133){$(1000\overline{1}10)$}
\put(61,135){$f_{44}$}
\put(53,135){\line(0,1){5}}
\put(54,137){$\alpha_{5}$}

\put(45,140){$(100\overline{1}100)$}
\put(61,142){$f_{46}$}
\put(53,142){\line(0,1){5}}
\put(54,144){$\alpha_{4}$}

\put(45,147){$(1\overline{1}0100\overline{1})$}
\put(61,149){$f_{48}$}
\put(32,155){\line(4,-1){16}}
\put(40,150){$\alpha_{2}$}
\put(53,149){\line(0,1){5}}
\put(54,151){$\alpha_{3}$}

\put(45,154){$(0\overline{1}00001)$}
\put(61,156){$f_{50}$}
\put(48,161){\line(3,-1){11}}
\put(54,160){$\alpha_{2}$}

\put(30,161){$(010\overline{1}001)$}
\put(46,163){$f_{52}$}
\put(38,164){\line(0,1){4}}
\put(39,165){$\alpha_{4}$}

\put(30,168){$(0001\overline{1}00)$}
\put(46,170){$f_{53}$}
\put(38,171){\line(0,1){4}}
\put(39,172){$\alpha_{5}$}

\put(30,175){$(00001\overline{1}0)$}
\put(46,177){$f_{54}$}
\put(38,178){\line(0,1){4}}
\put(39,179){$\alpha_{6}$}

\put(30,182){$(000001\overline{1})$}
\put(46,184){$f_{55}$}
\put(38,185){\line(0,1){4}}
\put(39,186){$\alpha_{7}$}

\put(30,189){$(0000001)$}
\put(46,191){$f_{56}$}
\end{picture}

\section{Majid's double-bosonization and Rosso's quantum shuffle}
Up to now,
we have obtained the inductive constructions of $U_q(\mathfrak g)$'s for all finite-dimensional complex simple Lie algebras $\mathfrak g$,
which can be expressed in a unified form by the following
\begin{proposition}
Let $T_{V}$ be a $p$-dimensional irreducible representation of $U_q(\mathfrak g)$
with highest weight $-\mu$,
$(a_{ij})_{n\times n}$ the Cartan matrix of $\mathfrak g$,
$\nu$ the weight of central element $c^{-1}$.
Then the resulted quantum group
$U(V^{\vee}(R',R_{21}^{-1}),\widetilde{U_q^{ext}(\mathfrak g)},V(R^{\prime},R))$
is of higher-one rank, whose Cartan matrix is obtained from $(a_{ij})_{n\times n}$ by adding a row and a column with:
$a_{i,n+1}=\frac{2(\alpha_{i},\mu)}{(\alpha_{i},\alpha_{i})},a_{n+1,i}=\frac{2(\mu,\alpha_{i})}{(\mu,\mu)+(\nu,\nu)},$
and $\nu$ is orthogonal to $\mu$ and $\alpha_{i},i=1,\cdots,n$.
\end{proposition}
\begin{proof}
Combining with the inductive constructions obtained in \cite{HH1,HH2} and the results in section 3 of this paper,
we know that the new additional group-like element $K_{n+1}$ is $(m^{+})^{p}_{p}c^{-1}$.
From the expression of $(m^{+})^{p}_{p}$ in each case,
we observe that the corresponding weight of $(m^{+})^{p}_{p}$ is $\mu$,
namely $(m^{+})^{p}_{p}=K_{\mu}$,
so $K_{n+1}=K_{\mu+\nu}$.
Every simple root vector $E_{j}$ usually locates in the minor entries $(m^{+})^{i}_{i+1}$,
where $(m^{+})^{i}_{i+1}=a_{i}E_{j}(m^{+})^{i+1}_{i+1}$,
$a_{i}\in k[q,q^{-1}]$.
According to the cross relations in Theorem \ref{cor1},
we obtain the following relations between $e^{i}$ and $e^{i-1}$.
$$
\left\{
\begin{array}{l}
e^{i}(m^{+})^{i}_{i+1}=\lambda R^{i}_{a}{}^{i}_{b}(m^{+})^{a}_{i+1}e^{b}
=\lambda R^{i}_{i}{}^{i}_{i}(m^{+})^{i}_{i+1}e^{i}+\lambda R^{i}_{i+1}{}^{i}_{i-1}(m^{+})^{i+1}_{i+1}e^{i-1},\\
e^{i}(m^{+})^{i+1}_{i+1}=\lambda R^{i+1}_{i+1}{}^{i}_{i}(m^{+})^{i+1}_{i+1}e^{i},\\
e^{i-1}(m^{+})^{i+1}_{i+1}=\lambda R^{i+1}_{i+1}{}^{i-1}_{i-1}(m^{+})^{i+1}_{i+1}e^{i-1}.
\end{array}
\right.
$$
Combining with $(m^{+})^{i}_{i+1}=a_{i}E_{j}(m^{+})^{i+1}_{i+1}$,
we get
$e^{i-1}=a_{i}\frac{R^{i+1}_{i+1}{}^{i-1}_{i-1}}{R^{i}_{i+1}{}^{i}_{i-1}}
(e^{i}E_{j}-\frac{R^{i}_{i}{}^{i}_{i}}{R^{i+1}_{i+1}{}^{i}_{i}}E_{j}e^{i})$.
Then $\text{weight}(e^{i-1})=\text{weight}(e^{i})+\alpha_{j}$,
combining with $e^{k}\cdot c^{-1}=\lambda^{-1}e^{k}$ for any $k$,
so $(\nu,\alpha_{j})=0$.
On the other hand,
we know that the explicit form of the additional $K_{\mu+\nu}$ in each inductive construction,
so we can obtain the specific form of $\nu$,
and observe that $\nu$ is orthogonal to $\mu$.
\end{proof}
With these explicit inductive constructions,
we know that both double bosonization \cite{majid1} and quantum shuffle in \cite{rosso}
have the same application for the inductive constructions of $U_q(\mathfrak g)$'s for the finite-dimensional complex simple Lie algebras.
We will give the specific form of weight $\mu$ and $\nu$,
and the Dynkin diagram in each case,
which makes the readers understand intuitively.
The imaginary line and the filled circle mean that the original Dynkin diagram extends to the new added simple root
with an arrow pointing to the shorter of the two roots in each extended Dynkin diagram.

$(1)$ The $A_{n}$ series.

Take $\mathfrak g=\mathfrak{sl}_{n}$ and $V$ the fundamental representation with
$\mu=-\lambda_{n-1}$,
where $\lambda_{n-1}=\sum\limits_{i=1}^{n-1}\frac{i}{n}\alpha_{i}=\frac{1}{n}\sum\limits_{i=1}^{n-1}\varepsilon_{i}-\frac{n-1}{n}\varepsilon_{n}$.
$\nu=\frac{1}{n}\sum\limits_{i=1}^{n}\varepsilon_{i}-\varepsilon_{n+1}$,
then we get $\mathfrak{sl}_{n+1}$ (\cite{HH1}).

\setlength{\unitlength}{1mm}
\begin{picture}(98,10)
\put(6,4){\circle{1}}
\put(6.5,4){\line(1,0){12}}
\put(3,0){$\varepsilon_{1}-\varepsilon_{2}$}
\put(5,6){$\alpha_{1}$}
\put(19,4){\circle{1}}
\put(19.5,4){\line(1,0){12}}
\put(18,6){$\alpha_{2}$}
\put(15,0){$\varepsilon_{2}-\varepsilon_{3}$}
\multiput(32.5,4)(3,0){4}{\line(1,0){2}}
\put(45.5,4){\line(1,0){12}}
\put(58,4){\circle{1}}
\put(58.5,4){\line(1,0){12}}
\put(57,6){$\alpha_{n-2}$}
\put(43,0){$\varepsilon_{n-2}-\varepsilon_{n-1}$}
\put(71,4){\circle{1}}
\multiput(71.5,4)(3,0){6}{\line(1,0){2}}
\put(70,6){$\alpha_{n-1}$}
\put(65,0){$\varepsilon_{n-1}-\varepsilon_{n}$}
\put(89,4){\circle*{1}}
\put(89,0){$\varepsilon_{n}-\varepsilon_{n+1}$}
\end{picture}

$(2)$ The $B_{n}$ series.

$B_{2}$ can be constructed from $A_{1}$ in \cite{HH1}.
Then starting from $\mathfrak g=\mathfrak{so}_{2n+1}$ and $V$  is the fundamental representation with $\mu=-\lambda_{n}$,
where $\lambda_{n}=\sum\limits_{i=1}^{n}\alpha_{i}=\varepsilon_{n}$.
$\nu=\varepsilon_{n+1}$,
then we get $\mathfrak{so}_{2n+3}$  (\cite{HH1}).

\setlength{\unitlength}{1mm}
\begin{picture}(98,6)
\put(6,4){\circle{1}}
\put(6.5,4.2){\line(1,0){12}}
\put(6.5,3.8){\line(1,0){12}}
\put(12,3){$<$}
\put(5,6){$\alpha_{1}$}
\put(6,0){$\varepsilon_{1}$}
\put(19,4){\circle{1}}
\put(19.5,4){\line(1,0){12}}
\put(18,6){$\alpha_{2}$}
\put(13,0){$\varepsilon_{2}-\varepsilon_{1}$}
\multiput(32.5,4)(3,0){4}{\line(1,0){2}}
\put(45.5,4){\line(1,0){12}}
\put(58,4){\circle{1}}
\put(58.5,4){\line(1,0){12}}
\put(57,6){$\alpha_{n-1}$}
\put(43,0){$\varepsilon_{n-1}-\varepsilon_{n-2}$}
\put(71,4){\circle{1}}
\multiput(71.5,4)(3,0){6}{\line(1,0){2}}
\put(70,6){$\alpha_{n}$}
\put(65,0){$\varepsilon_{n}-\varepsilon_{n-1}$}
\put(89,4){\circle*{1}}
\put(85,0){$\varepsilon_{n+1}-\varepsilon_{n}$}
\end{picture}

$(3)$ The $C_{n}$ series.

$C_{3}$ can be constructed from $A_{2}$ in \cite{HH1}.
Then starting from $\mathfrak g=\mathfrak{sp}_{2n}$ and $V$ is the fundamental representation with $\mu=-\lambda_{n}$,
where $\lambda_{n}=\frac{1}{2}+\sum\limits_{i=2}^{n}\alpha_{i}=\varepsilon_{n}$.
$\nu=\varepsilon_{n+1}$,
then we get $\mathfrak{sp}_{2n+2}$ (\cite{HH1}).

\setlength{\unitlength}{1mm}
\begin{picture}(98,6)
\put(6,4){\circle{1}}
\put(6.5,4.2){\line(1,0){12}}
\put(6.5,3.8){\line(1,0){12}}
\put(12,3){$>$}
\put(5,6){$\alpha_{1}$}
\put(6,0){$2\varepsilon_{1}$}
\put(19,4){\circle{1}}
\put(19.5,4){\line(1,0){12}}
\put(18,6){$\alpha_{2}$}
\put(13,0){$\varepsilon_{2}-\varepsilon_{1}$}
\multiput(32.5,4)(3,0){4}{\line(1,0){2}}
\put(45.5,4){\line(1,0){12}}
\put(58,4){\circle{1}}
\put(58.5,4){\line(1,0){12}}
\put(57,6){$\alpha_{n-1}$}
\put(43,0){$\varepsilon_{n-1}-\varepsilon_{n-2}$}
\put(71,4){\circle{1}}
\multiput(71.5,4)(3,0){6}{\line(1,0){2}}
\put(70,6){$\alpha_{n}$}
\put(65,0){$\varepsilon_{n}-\varepsilon_{n-1}$}
\put(89,4){\circle*{1}}
\put(85,0){$\varepsilon_{n+1}-\varepsilon_{n}$}
\end{picture}

$(4)$ The $D_{n}$ series.

$D_{4}$ can be constructed from $A_{3}$ in \cite{HH1}.
Then starting from $\mathfrak g=\mathfrak{so}_{2n}$ and $V$ is the fundamental representation with $\mu=-\lambda_{n}$,
where $\lambda_{n}=\frac{1}{2}(\alpha_{1}+\alpha_{2})+\sum\limits_{i=3}^{n}\alpha_{i}=\varepsilon_{n}$.
$\nu=\varepsilon_{n+1}$,
then we get $\mathfrak{so}_{2n+2}$ (\cite{HH1}).

\setlength{\unitlength}{1mm}
\begin{picture}(98,16)
\put(6,12){\circle{1}}
\put(18.5,8){\line(-3,1){12}}
\put(6,4){\circle{1}}
\put(6.5,4){\line(3,1){12}}
\put(0,13){$\alpha_{1}=\varepsilon_{2}-\varepsilon_{1}$}
\put(0,0){$\alpha_{2}=\varepsilon_{2}+\varepsilon_{1}$}
\put(19,8){\circle{1}}
\put(19.5,8){\line(1,0){12}}
\put(18,9){$\alpha_{3}$}
\put(13,4){$\varepsilon_{3}-\varepsilon_{2}$}
\multiput(32.5,8)(3,0){4}{\line(1,0){2}}
\put(45.5,8){\line(1,0){12}}
\put(58,8){\circle{1}}
\put(58.5,8){\line(1,0){12}}
\put(57,9){$\alpha_{n-1}$}
\put(43,4){$\varepsilon_{n-1}-\varepsilon_{n-2}$}
\put(71,8){\circle{1}}
\multiput(71.5,8)(3,0){6}{\line(1,0){2}}
\put(70,9){$\alpha_{n}$}
\put(65,4){$\varepsilon_{n}-\varepsilon_{n-1}$}
\put(89,8){\circle*{1}}
\put(85,4){$\varepsilon_{n+1}-\varepsilon_{n}$}
\end{picture}

$(5)$ The case of $E_{6}$.

Take $\mathfrak g=\mathfrak{so}_{10}$ and $V$ the $16$-dimensional half-spin representation with $\mu=-\lambda_{5}$,
where $\lambda_{5}=\frac{1}{2}\alpha_{1}+\alpha_{2}+\frac{3}{2}\alpha_{3}+\frac{3}{4}\alpha_{4}+\frac{5}{4}\alpha_{5}
=\frac{1}{2}\sum\limits_{i=1}^{5}\varepsilon_{i}$.
$\nu=\frac{1}{2}(\varepsilon_{7}+\varepsilon_{8}-\varepsilon_{6})$,
then we get $E_{6}$.

\setlength{\unitlength}{1mm}
\begin{picture}(60,23)
\put(6,4){\circle{1}}
\put(6.5,4){\line(1,0){12}}
\put(5,6){$\alpha_{1}$}
\put(0,0){$\varepsilon_{1}-\varepsilon_{2}$}

\put(19,4){\circle{1}}
\put(19.5,4){\line(1,0){12}}
\put(18,6){$\alpha_{2}$}
\put(13,0){$\varepsilon_{2}-\varepsilon_{3}$}

\put(32,4){\circle{1}}
\put(32.5,4){\line(1,0){12}}
\put(28,6){$\alpha_{3}$}
\put(26,0){$\varepsilon_{3}-\varepsilon_{4}$}

\put(32,4.5){\line(0,0){12}}
\put(32,17){\circle{1}}
\put(26,19.5){$\alpha_{4}=\varepsilon_{4}-\varepsilon_{5}$}

\put(45,4){\circle{1}}
\multiput(45.5,4)(3,0){4}{\line(1,0){2}}
\put(44,6){$\alpha_{5}$}
\put(39,0){$\varepsilon_{4}+\varepsilon_{5}$}

\put(58,4){\circle*{1}}
\put(53,0){$\frac{1}{2}(\varepsilon_{7}+\varepsilon_{8}-(\varepsilon_{1}+\cdots+\varepsilon_{6}))$}
\end{picture}

$(6)$ The case of $E_{7}$.

Take $\mathfrak g=E_{6}$ and $V$ the $27$-dimensional fundamental representation with $\mu=-\lambda_{6}$,
$\lambda_{6}=\frac{2}{3}\alpha_{1}+\alpha_{2}+\frac{4}{3}\alpha_{3}+2\alpha_{4}+\frac{5}{3}\alpha_{5}+\frac{4}{3}\alpha_{6}
=\frac{1}{3}\varepsilon_{8}-\frac{1}{3}(\varepsilon_{6}+\varepsilon_{7})+\varepsilon_{5},
$
and $\nu=\frac{1}{3}\varepsilon_{8}+\frac{2}{3}\varepsilon_{6}-\frac{1}{3}\varepsilon_{7}$,
then we get $E_{7}$.

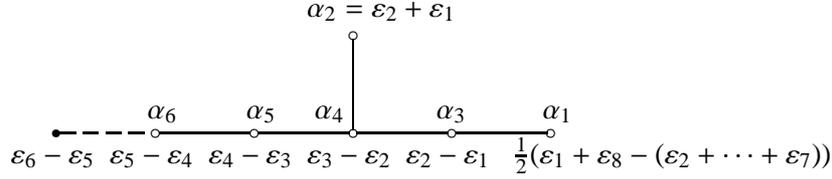
\setlength{\unitlength}{1mm}
\begin{picture}(73,25)
\put(6,4){\circle*{1}}
\multiput(6.5,4)(3,0){4}{\line(1,0){2}}
\put(0,0){$\varepsilon_{6}-\varepsilon_{5}$}

\put(19,4){\circle{1}}
\put(19.5,4){\line(1,0){12}}
\put(18,6){$\alpha_{6}$}
\put(13,0){$\varepsilon_{5}-\varepsilon_{4}$}

\put(32,4){\circle{1}}
\put(32.5,4){\line(1,0){12}}
\put(31,6){$\alpha_{5}$}
\put(26,0){$\varepsilon_{4}-\varepsilon_{3}$}

\put(45,4){\circle{1}}
\put(45.5,4){\line(1,0){12}}
\put(40,6){$\alpha_{4}$}
\put(39,0){$\varepsilon_{3}-\varepsilon_{2}$}

\put(45,4.5){\line(0,1){12}}
\put(45,17){\circle{1}}
\put(39,19.5){$\alpha_{2}=\varepsilon_{2}+\varepsilon_{1}$}

\put(58,4){\circle{1}}
\put(58.5,4){\line(1,0){12}}
\put(56,6){$\alpha_{3}$}
\put(52,0){$\varepsilon_{2}-\varepsilon_{1}$}

\put(71,4){\circle{1}}
\put(70,6){$\alpha_{1}$}
\put(66,0){$\frac{1}{2}(\varepsilon_{1}+\varepsilon_{8}-(\varepsilon_{2}+\cdots+\varepsilon_{7}))$}
\end{picture}

$(7)$ The case of $E_{8}$.

Take $\mathfrak g=E_{7}$ and $V$ the $56$-dimensional fundamental representation with $\mu=-\lambda_{7}$,
$\lambda_{7}=\alpha_{1}+\frac{3}{2}\alpha_{2}+2\alpha_{3}+3\alpha_{4}+\frac{5}{2}\alpha_{5}+2\alpha_{6}+\frac{3}{2}\alpha_{7}
=\frac{1}{2}\varepsilon_{8}-\frac{1}{2}\varepsilon_{7}+\varepsilon_{6},
$
and $\nu=\frac{1}{2}(\varepsilon_{8}+\varepsilon_{7})$,
then we get $E_{8}$.

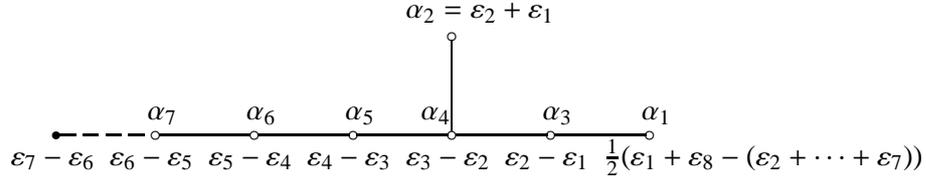
\setlength{\unitlength}{1mm}
\begin{picture}(88,25)
\put(6,4){\circle*{1}}
\multiput(6.5,4)(3,0){4}{\line(1,0){2}}
\put(0,0){$\varepsilon_{7}-\varepsilon_{6}$}

\put(19,4){\circle{1}}
\put(19.5,4){\line(1,0){12}}
\put(18,6){$\alpha_{7}$}
\put(13,0){$\varepsilon_{6}-\varepsilon_{5}$}

\put(32,4){\circle{1}}
\put(32.5,4){\line(1,0){12}}
\put(31,6){$\alpha_{6}$}
\put(26,0){$\varepsilon_{5}-\varepsilon_{4}$}

\put(45,4){\circle{1}}
\put(45.5,4){\line(1,0){12}}
\put(44,6){$\alpha_{5}$}
\put(39,0){$\varepsilon_{4}-\varepsilon_{3}$}

\put(58,4){\circle{1}}
\put(58.5,4){\line(1,0){12}}
\put(54,6){$\alpha_{4}$}
\put(52,0){$\varepsilon_{3}-\varepsilon_{2}$}

\put(58,4.5){\line(0,1){12}}
\put(58,17){\circle{1}}
\put(52,19.5){$\alpha_{2}=\varepsilon_{2}+\varepsilon_{1}$}

\put(71,4){\circle{1}}
\put(71.5,4){\line(1,0){12}}
\put(70,6){$\alpha_{3}$}
\put(65,0){$\varepsilon_{2}-\varepsilon_{1}$}

\put(84,4){\circle{1}}
\put(83,6){$\alpha_{1}$}
\put(78,0){$\frac{1}{2}(\varepsilon_{1}+\varepsilon_{8}-(\varepsilon_{2}+\cdots+\varepsilon_{7}))$}
\end{picture}

$(8)$ The case of $F_{4}$.

Take $\mathfrak g=\mathfrak {so}_{7}$ and $V$ the $8$-dimensional spin representation with $\mu=-\lambda_{3}$,
$\lambda_{3}=\frac{1}{2}\alpha_{1}+\alpha_{2}+\frac{3}{2}\alpha_{3}=\frac{1}{2}\varepsilon_{1}+\frac{1}{2}\varepsilon_{2}+\frac{1}{2}\varepsilon_{3}$,
and $\nu=\frac{1}{2}\varepsilon_{4}$,
then we get $F_{4}$  (see \cite{HH2}).

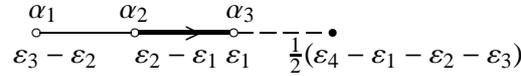
\setlength{\unitlength}{1mm}
\begin{picture}(88,8)
\put(66,4){\circle*{1}}
\multiput(53.5,4)(3,0){4}{\line(1,0){2}}
\put(26,6){$\alpha_{1}$}
\put(24,0){$\varepsilon_{3}-\varepsilon_{2}$}
\put(60,0){$\frac{1}{2}(\varepsilon_{4}-\varepsilon_{1}-\varepsilon_{2}-\varepsilon_{3})$}

\put(40,4){\circle{1}}
\put(40.5,3.8){\line(1,0){12}}
\put(40.5,4.2){\line(1,0){12}}
\put(46.5,3){$>$}
\put(52,6){$\alpha_{3}$}
\put(52,0){$\varepsilon_{1}$}

\put(53,4){\circle{1}}
\put(27.5,4){\line(1,0){12}}
\put(38,6){$\alpha_{2}$}
\put(40,0){$\varepsilon_{2}-\varepsilon_{1}$}

\put(27,4){\circle{1}}

\end{picture}

$(9)$ The case of $G_{2}$.

Take $\mathfrak g=\mathfrak {sl}_{2}$ and $V$ the $4$-dimensional spin $\frac{3}{2}$ representation with  $\mu=-3\lambda_{1}$,
$\lambda_{1}=\frac{1}{2}\alpha_{1}=\frac{1}{2}(\varepsilon_{1}-\varepsilon_{2})$,
and $\nu=-\frac{1}{2}(\varepsilon_{1}+\varepsilon_{2})+\varepsilon_{3}$,
then we get $G_{2}$ (see \cite{HH2}).

\setlength{\unitlength}{1mm}
\begin{picture}(88,6)
\put(8,4){\circle{1}}
\multiput(8.5,4)(3,0){7}{\line(1,0){2}}
\multiput(8.5,3.8)(3,0){7}{\line(1,0){2}}
\multiput(8.5,4.2)(3,0){7}{\line(1,0){2}}
\put(18,3){$<$}
\put(0,0){$\varepsilon_{1}-\varepsilon_{2}$}

\put(29,4){\circle*{1}}
\put(25,0){$-2\varepsilon_{1}+\varepsilon_{2}+\varepsilon_{3}$}
\end{picture}


\end{document}